\documentclass{amsart}
\DeclareSymbolFont{matha}{OML}{txmi}{m}{it}% txfonts
\DeclareMathSymbol{\varv}{\mathord}{matha}{118}
\usepackage{physics}
\usepackage{amsmath}%
\usepackage{amsfonts}%
\usepackage{amssymb}%
\usepackage{graphicx}
\newtheorem{theorem}{Theorem}
\theoremstyle{plain}

\newtheorem{corollary}{Corollary}

\newtheorem{lemma}{Lemma}

\newtheorem{proposition}{Proposition}
\newtheorem{remark}{Remark}

\numberwithin{equation}{section}
\DeclareSymbolFont{matha}{OML}{txmi}{m}{it}
\DeclareMathSymbol{\varv}{\mathord}{matha}{118}

\begin{document}
\title[On the collapse of the local Rayleigh condition]{On the collapse of the local Rayleigh condition for the Hydrostatic Euler equations and the finite time blow-up for the Semi-Lagrangian equations}
\author{Victor Ca\~nulef-Aguilar}

\date{October 29, 2022}

\begin{abstract}
In this paper we study the propagation of the local Rayleigh condition for the two-dimensional hydrostatic Euler equation in the framework of the local well-posedness result by Masmoudi and Wong \cite{MaTKW12}. We show under certain assumptions that such solutions will develop singularities or collapse the local Rayleigh condition. In addition, we find necessary conditions for the global solvability. Finally, we establish the finite time blow-up of solutions to the semi-lagrangian equations introduced by Brenier in \cite{Bre99} for certain class of initial data.
\end{abstract}

\maketitle

\maketitle

\section{Introduction}

We consider the two-dimensional hydrostatic Euler equation in a periodic channel $\Omega=\mathbb{T}\times (0,1)$:

\begin{align} \label{HEE}
\left\{ \begin{aligned}
 u_{t}+uu_{x}+vu_{y}+P_{x}&=0,& &\\
u_{x}+v_{y}&=0,& & \\
P_{y}&=0,& & \\
v\rvert_{y=0,1}&=0,& & \\
u\rvert_{t=0}&=u_{0},& &
\end{aligned} \right. 
\end {align}
where $(u,v)$ is the velocity field, $P$ is the scalar pressure and $u$ satisfies the local Rayleigh condition (i.e. $\partial_y^2 u>0$). Equations (\ref{HEE}) are derived from the two-dimensional incompressible Euler equation by means of the hydrostatic approximation. The mathematical justification of the formal limit has been established in several articles under the local Rayleigh condition (see for instance \cite{Bre03}, \cite{Gre99} and \cite{MaTKW12}). On the other hand, if the initial condition does not satisfies the local Rayleigh condition, the convergence may not hold, as shown in \cite{Gre00}, \cite{Gre99}. Furthermore, the ill-posedness of the linearization of (\ref{HEE}) around certain shear flows was proved in \cite{Re09}, as well as the ill-posedness of (\ref{HEE}) around certain shear flows, which was proved in \cite{H-KN16}.

The local existence of solutions to (\ref{HEE}) has been proven in \cite{KTVZ10} and \cite{KTVZ11} in the analytic setting. Nevertheless, the finite time blow-up of solutions to (\ref{HEE}) was established in \cite{TKW12} (under the assumption $u_0(x_0, y,t)\equiv C$, for $ y\in [0,1]$, which is incompatible with the local Rayleigh condition) and \cite{CINT12} (under the assumption of parity, which is again incompatible with the local Rayleigh condition. See also \cite{EE97}, where the finite time blow-up of certain solutions to the unsteady Prandtl equations is proved by similar methods). 

Equivalently, we may consider the vorticity formulation of the hydrostatic Euler equations (\ref{HEE})

\begin{align} \label{VHEE}
\left\{ \begin{aligned}
 \omega_{t}+u\omega_{x}+v\omega_{y}&=0,& &\\
(u,v) &=\nabla^{\perp}\mathcal{A}\omega,& & \\
\omega\rvert_{t=0}&=\omega_{0},& & \\
\end{aligned} \right. 
\end {align}
where the vorticity $\omega$ is defined by

\begin{equation}
\omega=u_y.
\end{equation}
The local well-posedness of (\ref{VHEE}) was proved in \cite{MaTKW12} in Sobolev spaces $H^s$, with $s\geq 4$, for initial vorticity profiles that satisfy the local Rayleigh condition (i.e. $\partial_y\omega_0 >0$). In this article we address the propagation in time of the local Rayleigh condition; we get lower bounds for certain functionals that quantify in some sense the validity of the local Rayleigh condition (see (\ref{collapselogE1}) and (\ref{collapselogE2})). Under certain assumptions, we will prove that the above functionals cannot remain bounded, which implies the collapse of the local Rayleigh condition or the formation of singularities (see Theorem \ref{collapse}, which is proved in Section \ref{proofcollapse}). One of the main achievements of this work, is the derivation of certain identities which are satisfied by every solution to (\ref{VHEE}), as long as the solution exists (see Proposition \ref{formulas} in Section \ref{monoticitysection}). In addition, we derive some necessary conditions for the global solvability in the framework of the local well-posedness result by Masmoudi and Wong \cite{MaTKW12} (see Theorem \ref{properties} and Section \ref{proofproperties}). Despite Theorem \ref{collapse} does not guarantees that the solution remains in $H^1(\mathbb{T}\times (0,1))$ as long as the local Rayleigh condition holds, we can control the $H^1$ norm and the validity of the local Rayleigh condition under additional assumptions, which is presented in Section \ref{propagation}.

The second main result of this article deals with the finite time blow-up of smooth solutions to the semilagrangian equations, which are derived from the hydrostatic Euler equations (\ref{VHEE}) under certain assumptions. More precisely, if $\omega_0$ satisfies the local Rayleigh condition and 

\begin{align}
\omega_0(x,0)&\equiv k,\label{constantvorticity1}\\
\omega_0(x,1)&\equiv k+1,\label{constantvorticity2}
\end{align}
for certain constant $k$, then problem (\ref{VHEE}) is equivalent to the following problem in $\Omega=\mathbb{T}\times (0,1)$:

\begin{align}\label{SLE1D}
\left\{ \begin{aligned}
 \varv_{t}+\partial_x\left(\frac{\varv^2}{2}+P\right)&=0,& &\\
\partial_t h_a+\partial_x\left(\varv h_a\right)&=0,& &\\
\int_{0}^{1}h_ada&=1,& &\\
\partial_a P&=0,& & \\
(\nu,h_a)\rvert_{t=0}&=(\nu_0,\partial_a h_0),& & \\
\end{aligned} \right. 
\end{align}
where $\varv$ and $h_a$ are defined by
\begin{align}
  \omega(x,h(x,t,a),t)&\equiv k+a,\text{\hspace{0.1cm} for $a\in[0,1]$,}\label{SLE1Ddefh}\\
  \varv(x,t,a)&=u(x,h(x,t,a),t),\text{\hspace{0.1cm} for $a\in[0,1]$,}\label{SLE1Ddefv}
  \end{align}
in this case we assume that $\partial_a h_0$ is positive, which is equivalent to the validity of the local Rayleigh condition. Equations (\ref{SLE1D}) are known as the semilagrangian equations, which were introduced by Brenier in \cite{Bre99}. In the same article, the local existence of solutions to (\ref{VHEE}) was proved in the class $C$ (which is defined as the set of $C^1$-functions in $\mathbb{T}\times (0,1)$ that satisfy the local Rayleigh condition, (\ref{constantvorticity1}) and (\ref{constantvorticity2})). The semilagrangian equations (\ref{SLE1D}) have a natural extension to higher dimensions, namely

\begin{align} \label{HSLE}
\left\{ \begin{aligned}
 \partial_t \vb{v}+\vb{v}\cdot \nabla \vb{v}+\nabla P&=0,& & \\
\partial_t h_a+\nabla \cdot \left(\vb{v}h_a\right)&=0,& & \\
\int_{0}^{1}h_ada&=1,& & \\
\partial_a P&=0,& & \\
\partial_i\varv_j-\partial_j\varv_i&=0,& &\text{for $ i,j\leq d$,} \\
(\vb{v},h_a)\rvert_{t=0}&=(\vb{v}_{0},\partial_ah_0),& &
\end{aligned} \right. 
\end {align}
where $\vb{v}: \mathbb{T}^d\times [0,1]\rightarrow \mathbb{R}^d$ is the semilagrangian velocity, $P$ is the scalar pressure and $h_a$ is a positive function in  $\mathbb{T}^d\times [0,1]$. Let us point out that the semilagrangian equations (\ref{SLE1D}) establish a vertical change of coordinates by means of the level curves of the vorticity $\omega$, where the injectivity follows by the local Rayleigh condition. Hence, the extension of the semilagrangian equations to higher dimensions have no clear physical meaning.

A necessary condition for the global solvability of (\ref{HSLE}) was obtained in \cite{Bre99}, which is given by:

\begin{equation}\label{BCC}
\int_{0}^{1}\left\lvert\int_{\mathbb{T}^d}\vb{v}(x',t,a)dx'\right\rvert^2\int_{\mathbb{T}}h_a(x,t,a)dxda=\int_{\mathbb{T}\times (0,1)}\lvert \vb{v}\rvert^2(x,t,a)h_a(x,t,a)dxda,
\end{equation}
where both sides are time independent. Theorem \ref{blowup} establishes the finite time blow-up of solutions to the semilagrangian equations (\ref{HSLE}) for certain class of initial data, which is proved in Section \ref{proofblowup}. In the following subsections, we will state the main results of this article.

\subsection{Collapse of the local Rayleigh condition or singularity formation}

\begin{theorem}\label{collapse}

Assume that the initial vorticity $\omega_0\in H^4(\mathbb{T}\times(0,1)) $ satisfies the local Rayleigh condition (i.e. $\partial_y\omega_0>0$) and $\partial_x\omega_0$ is not identically $0$. Let $\omega\in C([0,T];H^4(\mathbb{T}\times(0,1)))$ be the solution to the problem (\ref{VHEE}) that satisfies the local Rayleigh condition in $[0,T]$. Set

\begin{equation}\label{defE1}
E_1(t)=\int_{\mathbb{T}\times(0,1)}\frac{\omega\omega_x}{\omega_y}dxdy = \int_{\mathbb{T}\times(0,1)}\left(\frac{\omega\omega_x}{\omega_y}-u_x\right)dxdy
\end{equation} and

\begin{equation}\label{defE2}
E_2(t)=\int_{\mathbb{T}\times(0,1)}u^2\frac{\omega\omega_x}{\omega_y}dxdy = \int_{\mathbb{T}\times(0,1)}u^2\left(\frac{\omega\omega_x}{\omega_y}-u_x\right)-uP_xdxdy.
\end{equation}

Assume that $E_1(0)\geq 0$ or $E_2(0)\geq 0$. Then there exists a positive time $T^{*}$ such that at least one of the following properties holds:

\begin{itemize}

\item $\displaystyle \limsup_{t\rightarrow T^{*}}\Vert \omega (t)\Vert_{H^4(\mathbb{T}\times (0,1))}=\infty,$

\item $\displaystyle\limsup_{t\rightarrow T^{*}}\left\Vert \frac{1}{\omega_y (t)}\right\Vert_{L^{\infty}(\mathbb{T}\times (0,1))}=\infty.  $

\end{itemize}
Furthermore, as long as $\Vert \omega (t)\Vert_{H^4(\mathbb{T}\times (0,1))}$ and $\left\Vert\frac{1}{\omega_y(t)}\right\Vert_{L^{\infty}(\mathbb{T}\times (0,1))}$ remain bounded, we have the following estimates: 
   
\begin{equation}\label{collapselogE1}
\log \left(\frac{1}{E_1(0)}\frac{1}{\frac{1}{E_1(0)}-t}\right)\leq \int_{\mathbb{T}\times (0,1)}\log\left(\frac{\partial_y\omega_0}{\omega_y(t)}\right)dxdy,\text{\hspace{1.7cm}if $E_1(0)>0,$}\end{equation}

\begin{equation}\label{collapselogE2}
\log \left(\frac{\Vert u\Vert_2^2}{E_2(0)}\frac{1}{\frac{\Vert u\Vert_2^2}{E_2(0)}-t}\right)\leq tC_1(\omega_0)+ C_2(\omega_0)\int_{\mathbb{T}\times (0,1)}\log\left(\frac{\partial_y\omega_0}{\omega_y(t)}\right)dxdy,\text{\hspace{0.1cm}if $E_2(0)>0,$}\end{equation}

\begin{equation}\label{collapseE1}
\frac{1}{\frac{1}{E_1(0)}-t}\leq E_1(t)\leq E_1(0)+C_3(\omega_0)\int_{0}^{t}\int_{\mathbb{T}\times (0,1)}\omega_x^2\left(1+\frac{1}{\omega_y^2}\right)dxdyd\tau, \text{\hspace{0.1cm}if $E_1(0)>0,$} 
\end{equation}

\begin{equation}\label{collapseE2}
\frac{\Vert u\Vert_{2}^{2}}{\frac{\Vert u\Vert_{2}^{2}}{E_2(0)}-t}\leq E_2(t)\leq E_2(0)+C_4(\omega_0)\int_{0}^{t}\int_{\mathbb{T}\times (0,1)}\omega_x^2\left(1+\frac{1}{\omega_y^2}\right)dxdyd\tau, \text{\hspace{0.05cm}if $E_2(0)>0,$} 
\end{equation}
where $C_1(\omega_0)=\left(\frac{3}{2}\frac{\Vert \omega_0\Vert_{\infty}}{\Vert u\Vert_2}\right)^4\lvert E_1(0) \rvert+\frac{E_2(0)}{\Vert u\Vert_2^2}$, $C_2(\omega_0)=\left(\frac{3}{2}\frac{\Vert \omega_0\Vert_{\infty}}{\Vert u\Vert_2}\right)^4,$ $C_3(\omega_0)=2\Vert\omega_0\Vert_{\infty}^2+\frac{2}{\pi^2}$ and $C_4(\omega_0)=2\left(\frac{3}{\pi}\Vert \omega_0\Vert_{\infty}\right)^2+2\left(\frac{3}{2}\Vert \omega_0\Vert_{\infty}\right)^2C_3(\omega_0)$.

\end{theorem}

\begin{remark}

If $\omega_0\in H^4(\mathbb{T}\times (0,1))$ satisfies the local Rayleigh condition and $\partial_x\omega_0\equiv 0$, then the solution to (\ref{VHEE}) is stationary. Conversely, the stationary solutions in $H^4(\mathbb{T}\times(0,1))$ that satisfy the local Rayleigh condition are $x$-independent (see Corollary \ref{stationary}).

\end{remark}

\begin{remark}
Suppose that $\omega_0\in H^4(\mathbb{T}\times (0,1))$ is even in $x$, satisfies the local Rayleigh condition and  $\partial_x \omega_0$ is not identically $0$, then $E_1(0)=0$, which implies the collapse in finite time (more precisely, the local Rayleigh condition will collapse or $\omega$ will develop singularities) of the solution to (\ref{VHEE}). Moreover, since $\omega(x,y,t)$ and $\omega(-x,y,-t)$ solve (\ref{VHEE}) with the same initial condition, $\omega(x,y,t)\equiv \omega(-x,y,-t)$ (thanks to the uniqueness). In that case, the solution will collapse going forward and backward in time.

\end{remark}

\begin{remark}
It is worth mentioning that the set of initial conditions that satisfy the assumptions of Theorem \ref{collapse} is nonempty. Indeed, if we choose $\omega_0(x,y)=2y-sin(2\pi x-y)$, then $\partial_x\omega_0=-2\pi cos(2\pi x-y)$ and $\partial_y\omega_0=2+cos(2\pi x-y)$, which yields

\begin{align}
&E_1(0)\nonumber\\
&=\int_{\mathbb{T}\times (0,1)}\frac{(2y-sin(2\pi x-y))\cdot-2\pi cos(2\pi x-y)}{2+cos(2\pi x-y)}dxdy\nonumber\\
&=-2\pi\int_{\mathbb{T}\times (0,1)}\left(2y-sin(2\pi x-y)\right)dxdy+4\pi\int_{\mathbb{T}\times (0,1)}\frac{(2y-sin(2\pi x-y))}{2+cos(2\pi x-y)}dxdy\nonumber\\
&=-2\pi+\int_{\mathbb{T}\times (0,1)}\frac{4\pi\cdot 2y}{2+cos(2\pi x-y)}dxdy\nonumber\\
&=-2\pi+\frac{4\pi}{\sqrt{3}}\int_{0}^{1}2y dy\nonumber\\
&=2\pi\left(\frac{2}{\sqrt{3}}-1\right)\nonumber\\
&>0.\nonumber
\end{align}

\end{remark}

\begin{remark}
Suppose that $\omega_0(x,y)\in H^4(\mathbb{T}\times (0,1))$ satisfies the local Rayleigh condition, $\partial_x\omega_0$ is not identically $0$ but $E_1(0)<0$, then $\tilde{\omega}_0(x,y)=\omega_0(-x,y)$ satisfies $E_1(0)>0$ (the same applies for $E_2$). Hence, the assumptions of Theorem \ref{collapse} are not so restrictive.

\end{remark}

\subsection{Necessary conditions for global solvability}

\begin{theorem}\label{properties}

Assume that the initial vorticity $\omega_0\in H^4(\mathbb{T}\times(0,1)) $ satisfies the local Rayleigh condition (i.e. $\partial_y\omega_0>0$). Let $\omega\in C([0,T];H^4(\mathbb{T}\times(0,1)))$ be the solution to the problem (\ref{VHEE}) that satisfies the local Rayleigh condition in $[0,T]$. Let $E_1$ and $E_2$ be defined by (\ref{defE1}) and (\ref{defE2}) respectively. Assume further that $T$ can be chosen arbitrarily large, then we have:

\begin{align}
&\int_{0}^{\infty}\int_{\mathbb{T}\times(0,1)}\left\lvert P_x+u\left(u_x-\frac{\omega\omega_x}{\omega_y}\right)\right\rvert^2dxdydt=-E_2(0),\label{properties1} \\
&\int_{0}^{\infty}\int_{\mathbb{T}\times(0,1)}\left\lvert u_x-\frac{\omega\omega_x}{\omega_y} \right\rvert^2dxdydt=-E_1(0),\label{properties2} \\
&\int_{0}^{\infty}\int_{\mathbb{T}^1}P_x^2dxdt\leq -2E_2(0)-\frac{9}{2}\Vert \omega_0\Vert_{\infty}^2E_1(0),\label{properties3} \\
&\int_{0}^{\infty}\lvert E_1(t)\rvert dt\leq \int_{\mathbb{T}\times (0,1)}\log\left(\frac{2\Vert\omega_0\Vert_{\infty}}{\partial_y\omega_0}\right)dxdy,\label{properties4} 
\end{align}

where $E_1(0),E_2(0)<0$.
\end{theorem}

\begin{remark}
Let us observe that in the semilagrangian formulation (\ref{SLE1D}), the above energies are equal to (see Proposition \ref{summarySLE1D}):

\begin{align}
E_1(T)&=-\int_{\mathbb{T}\times (0,1)}\varv_x h_{a}dxda=E_1(0)+\int_{0}^{T}\int_{\mathbb{T}\times (0,1)}\varv_{x}^{2}h_{a} dxda dt,\nonumber\\
E_2(T)&=-\int_{\mathbb{T}\times (0,1)}\varv^2\varv_x h_{a}dxda=E_2(0)+\int_{0}^{T}\int_{\mathbb{T}\times (0,1)}\varv_{t}^{2}h_{a} dxda dt,\nonumber
\end{align}
thanks to the change of coordinates $y=h(x,t,a)$. The above identities are a particular case of (\ref{formulasSLE2}) and (\ref{formulasSLE3}), when $d=1$.

\end{remark}

\subsection{Finite time blow-up of the Semi-lagrangian equations}

\begin{theorem}\label{blowup}
Let $(\vb{v},h_a)$ be a smooth solution to (\ref{HSLE}). Set 

\begin{equation}\label{defE1HSLE}
E_1(t)=-\int_{\mathbb{T}^d\times (0,1)}\left(\nabla\cdot \vb{v} \right)h_adxda
\end{equation}
and

\begin{equation}\label{defE2HSLE}
E_2(t)=\int_{\mathbb{T}^d\times (0,1)}\left(\vb{v}\cdot\vb{v}_t \right)h_adxda.
\end{equation}

Assume that $E_1(0)>0$ or $E_2>0$, then the solution blows up in finite time.
Moreover, as long as the solution exists, we have the following estimates:

\begin{align}
d\log \left(\frac{d}{E_1(0)}\frac{1}{\frac{d}{E_1(0)}-t}\right)&\leq \int_{\mathbb{T}^d\times (0,1)}\log (h_a(\tau))h_a(\tau)dxda\rvert_{\tau=0}^{\tau=t},\text{\hspace{0.1cm}if $E_1(0)>0,$}\label{collapselogha} \\
\frac{d}{\frac{d}{E_1(0)}-t}\leq E_1(t)&=E_1(0)+\int_{0}^{t}\int_{\mathbb{T}^d\times (0,1)}\lvert\nabla\vb{v}\rvert^2h_adxdadt, \text{\hspace{0.1cm}if $E_1(0)>0,$} \label{collapseE1SLE}\\
\frac{\Vert \vb{v}\Vert^{2}}{\frac{\Vert \vb{v}\Vert^{2}}{E_2(0)}-t}\leq E_2(t)&=E_2(0)+\int_{0}^{t}\int_{\mathbb{T}^d\times (0,1)}\lvert\partial_t\vb{v}\rvert^2h_adxdadt, \text{\hspace{0.05cm}if $E_2(0)>0,$} \label{collapseE2SLE}
\end{align}

where $\Vert \vb{v}\Vert^2=\int_{\mathbb{T}^d\times (0,1)}\lvert \vb{v}\rvert^2h_adxda$.

\end{theorem}
The following corollary is a direct consequence of Theorem \ref{blowup}, Jensen's inequality and Lemma \ref{loglemma} (applied to the probability measure $\mu=h_adxda$).

\begin{corollary}\label{blowup2}
Let $(\vb{v},h_a)$ be a smooth solution to (\ref{HSLE}) in $[0,T]$. Let $E_1$ be defined by (\ref{defE1HSLE}). Suppose that $E_1(0)>0$, then

\begin{align}
&\exp\left(\int_{\mathbb{T}^d\times (0,1)}\partial_ah_0\log \left(\partial_ah_0\right)dxda\right) \left(\frac{d}{E_1(0)}\frac{1}{\frac{d}{E_1(0)}-t}\right)^d\nonumber\\
&\leq\exp\left(\int_{\mathbb{T}^d\times (0,1)}h_a(T)\log \left(h_a(T)\right)dxda\right)\nonumber\\
&\leq \left(\int_{\mathbb{T}^d\times (0,1)}h_a(T)^{1+p}dxda\right)^{1/p},\nonumber
\end{align}
for every $p\in(0,\infty)$. Moreover

\begin{equation}
\exp\left(\int_{\mathbb{T}^d\times (0,1)}h_a\log (h_a)dxda\right)=\lim_{p\rightarrow 0}\left( \int_{\mathbb{T}^d\times (0,1)}h_a^{1+p}dxda  \right)^\frac{1}{p}.
\end{equation}

\end{corollary}

\begin{remark}
Corollary \ref{blowup2} gives a lower bound for the $L^p(\mathbb{T}^d\times (0,1))$ norm of $h_a$ for $p>1$, which cannot be extended to $p=1$, because $\int_{\mathbb{T}^d\times (0,1)}h_adxda=1$.

\end{remark}

\begin{remark}

The first equation in (\ref{HSLE}), is equivalent to: 

\begin{equation}\label{burgersHSLE}
\partial_t \vb{v}+\nabla\left(\frac{\lvert \vb{v}\rvert^2 }{2}+P\right)=0,
\end{equation}
thanks to the  curl free condition $\partial_i\varv_j=\partial_j\varv_i$.

\end{remark}

\begin{remark} In particular, smooth solutions to (\ref{HSLE}) satisfy the following equations in $\mathbb{T}^d$:

\begin{align} 
\left\{ \begin{aligned}
 \partial_t\int_{0}^{1} \vb{v}h_ada+\nabla\cdot \int_{0}^{1} \vb{v}\otimes \vb{v} h_ada+\nabla P&=0,& & \\
 \int_{0}^{1}h_ada &=0,& & \\
\nabla \cdot \int_{0}^{1}\vb{v}h_ada &=0,& & \\
\partial_i\varv_j-\partial_j\varv_i&=0,& &\text{for $ i,j\leq d$} ,& &
\end{aligned} \right. 
\end {align}

which can be seen as an averaged version of the incompressible Euler equations in $\mathbb{T}^d$

\begin{align} 
\left\{ \begin{aligned}
 \partial_t\vb{v}+\nabla\cdot \left( \vb{v}\otimes \vb{v}\right) +\nabla P&=0,& & \\
\nabla \cdot \vb{v} &=0,& &\\
\vb{v}\rvert_{t=0} &=\vb{v}_0.& & 
\end{aligned} \right. 
\end {align}

\end{remark}

\section{Preliminaries}

In this section we will present some elementary results concerning equations (\ref{HEE}), (\ref{VHEE}) and (\ref{HSLE}) that will be used throughout this article.

\subsection{Hydrostatic Euler equations} 

First, let us point out that any solution to (\ref{HEE}) satisfies $\int_{0}^{1}u(x,y,t)dy\equiv k$ (see Proposition \ref{propelementary}). Moreover,  $\tilde{u}(x,y,t)=u(x+kt,y,t)-k$ solves (\ref{HEE}) with a slightly different initial data and $\int_{0}^{1}\tilde{u}(x,y,t)dy\equiv 0$. Hence, without loss of generality, we may assume that the solution to (\ref{HEE}) satisfies 

\begin{equation}\label{intzero}
\int_{0}^{1}u(x,y,t)dy\equiv 0.
\end{equation}

The following proposition summarizes some properties that will be used in the rest of the article.

\begin{proposition}\label{propelementary}
Assume that the initial vorticity $\omega_0\in H^4(\mathbb{T}\times(0,1)) $ satisfies the local Rayleigh condition (i.e. $\partial_y\omega_0>0$). Let $\omega\in C([0,T];H^4(\mathbb{T}\times(0,1)))$ be the solution to the problem (\ref{VHEE}) that satisfies the local Rayleigh condition in $[0,T]$. Then we have

\begin{align}
&\mathcal{A}(\omega)=(1-y)\int_{0}^{y}z\omega(x,z,t)dz+y\int_{y}^{1}(1-z)\omega(x,z,t)dz,\label{eqstream}\\
&u(x,y,t)=\int_{0}^{1}z\omega(x,z,t)dz-\int_{y}^{1}\omega(x,z,t)dz,\label{equ}\\
&v(x,y,t)=(1-y)\int_{0}^{y}z\omega_x(x,z,t)dz+y\int_{y}^{1}(1-z)\omega_x(x,z,t)dz,\label{eqv}\\
&\partial_x\int_{0}^{1}udy=0,\label{intux}\\
&\partial_t \int_{0}^{1}udy=0,\label{intut}\\
&P_x(x,t)=-\partial_x\left(\int_{0}^{1}u^2(x,y,t)dy\right),\label{defpressure}\\
 &\partial_t \Vert u\Vert_{2}^2=0,\label{kinetic}\\
&\Vert u\Vert_{\infty}\leq \frac{3}{2}\Vert \omega_0\Vert_{\infty}.\label{uinfty}
\end{align}

\end{proposition}

\begin{proof}

By (\ref{intzero}), we may assume that the stream function $\mathcal{A}(\omega)$ solves

\begin{align} \label{SHEE}
\left\{ \begin{aligned}
 -\partial_{y}^{2}\mathcal{A}(\omega)&=\omega,& & \\
\mathcal{A}(\omega)\rvert_{y=0,1}&=0,& &
\end{aligned} \right. 
\end {align}
which implies

\begin{align}
 \mathcal{A}(\omega)&=-\int_{0}^{y}\int_{0}^{z}\omega(x,s,t)dsdz+y\int_{0}^{1}\int_{0}^{z}\omega(x,s,t)dsdz \nonumber\\
 &=(1-y)\int_{0}^{y}z\omega(x,z,t)dz+y\int_{y}^{1}(1-z)\omega(x,z,t)dz,\nonumber
 \end{align}
from which we get (\ref{eqstream}). Next, (\ref{equ}) and (\ref{eqv}) follow directly from (\ref{eqstream}). The incompressibility condition and the boundary value in (\ref{HEE}) implies (\ref{intux}). Next, by (\ref{intux}), (\ref{HEE}), the $x$-periodicity and integration by parts,
\begin{align}
\partial_t \int_{0}^{1}udy&=\partial_t\int_{\mathbb{T}\times (0,1)}udxdy\nonumber\\
&=-\int_{\mathbb{T}\times (0,1)}\partial_x\left(\frac{u^2}{2}+P\right)+v\omega dxdy\nonumber\\
&=-\int_{\mathbb{T}\times (0,1)}uu_xdxdy\nonumber\\
&=-\int_{\mathbb{T}\times (0,1)}\partial_x\left(\frac{u^2}{2}\right)dxdy\nonumber\\
&=0,
\end{align} 
 which proves (\ref{intut}). Let us prove (\ref{defpressure}). By (\ref{HEE}), (\ref{intut}) and integration by parts, we have

\begin{align}
0&=\int_{0}^{1}\left(u_t+uu_x+vu_y+P_x\right)dy\nonumber\\
&=\int_{0}^{1}\left(uu_x+vu_y\right)dy+P_x\nonumber\\
&= 2\int_{0}^{1}uu_xdy+P_x,\nonumber
\end{align}
which implies (\ref{defpressure}). By (\ref{HEE}), the $x$-periodicity and integration by parts

\begin{align}
\partial_t\int_{\mathbb{T}\times (0,1)}u^2dxdy&=2\int_{\mathbb{T}\times (0,1)}uu_tdxdy\nonumber\\
&=-2\int_{\mathbb{T}\times (0,1)}(u^2u_x+vuu_y+uP_x)dxdy\nonumber\\
&=-3\int_{\mathbb{T}\times (0,1)}u^2u_xdxdy+2\int_{\mathbb{T}}\left(\int_{0}^{1}u_xdy\right)P(x,t)dx\nonumber\\
&=0,\nonumber
\end{align}
from which we get (\ref{kinetic}). Applying (\ref{equ}), we get

\begin{align}
\Vert u\Vert_{\infty}&\leq \Vert \omega\Vert_{\infty}\int_{0}^{1}zdz+\Vert \omega\Vert_{\infty}\nonumber\\
&= \frac{3}{2}\Vert \omega\Vert_{\infty}\nonumber\\
&= \frac{3}{2}\Vert \omega_0\Vert_{\infty}.\nonumber
\end{align}
This completes the proof of Proposition \ref{propelementary}.
\end{proof}

\subsection{Semilagrangian equations}

The following proposition summarizes how the change of variable works between the hydrostatic Euler equations (\ref{VHEE}) and the semilagrangian equations (\ref{SLE1D}).
\begin{proposition}\label{summarySLE1D} 
Let $\omega$ be a smooth solution to (\ref{VHEE}) that satisfies the local Rayleigh condition, (\ref{constantvorticity1}) and (\ref{constantvorticity2}) in $[0,T]$. Let $(\varv, h_a)$ be a smooth solution to (\ref{SLE1D}) in $[0,T]$. Assume that $\omega$, $h$, $u$ and $\varv$ satisfy (\ref{SLE1Ddefh}) and (\ref{SLE1Ddefv}) in $[0,T]$, then  
\begin{align}
&h_a(x,t,a)=\frac{1}{\omega_y(x,h(x,t,a),t)},\label{summarySLE1Dha}\\
&h_x(x,t,a)=-\frac{\omega_x}{\omega_y}(x,h(x,t,a),t),\label{summarySLE1Dhx}\\
&\varv_a(x,t,a)=(k+a)h_a,\label{summarySLE1Dva}\\
&\varv_x(x,t,a)=\left(u_x-\frac{\omega\omega_x}{\omega_y}\right)(x,h(x,t,a),t),\label{summarySLE1Dvx}\\
&\varv(x,t,a) =-\frac{k+1}{2}+(k+a)h-\frac{1}{2}\int_{0}^{1}h^2db+\int_{a}^{1}hdb, \label{summarySLE1Dv}\\
&\mathcal{A}(\omega)(x,h(x,t,a),t)=-\varv(x,t,a) h(x,t,a)+\frac{1}{2}(k+a)h^2-\frac{1}{2}\int_{0}^{a}h^2(x,t,b)db,\label{summarySLE1Dstream}\\
&h_t(x,t,a)-\partial_x\left(\mathcal{A}(\omega)(x,h(x,t,a),t)\right)=0,\label{summarySLE1Dht}\\
&\varv_t(x,t,a)=\left(u\left(\frac{\omega\omega_x}{\omega_y}-u_x\right)(x,h(x,t,a),t)-P_x(x,t)\right),\label{summarySLE1Dvt}
\end{align}
where $0\leq t\leq T$.
\end{proposition}
\begin{proof}

First observe that (\ref{summarySLE1Dha}), (\ref{summarySLE1Dhx}), (\ref{summarySLE1Dva}) and (\ref{summarySLE1Dvx}) follow directly by (\ref{SLE1Ddefh}) and (\ref{SLE1Ddefv}). Applying (\ref{equ}) and the change of variable $z=h(x,t,b)$

\begin{align}
u(x,h,t)&=\int_{0}^{1}z\omega(x,z,t)dz-\int_{h(	x,t,a)}^{1}\omega(x,z,t)dz\nonumber\\
&=\int_{0}^{1}(k+b)hh_bdb-\int_{a}^{1}(k+b)h_bdb\nonumber\\
&=-\frac{k+1}{2}-\frac{1}{2}\int_{0}^{1}h^2db+(k+a)h+\int_{a}^{1}hdb,\nonumber
\end{align}
which proves (\ref{summarySLE1Dv}). Next, by (\ref{VHEE}), we have $-\mathcal{A}(\omega)(x,y,t)=\int_{0}^{y}u(x,z,t)dz$, which implies

\begin{align}
-\mathcal{A}(\omega)(x,h,t)&=\int_{0}^{h(x,t,a)}u(x,z,t)dz\nonumber\\
&=\int_{0}^{a}\varv (x,t,b)h_b(x,t,b)db\nonumber\\
&=\varv(x,t,a) h(x,t,a)-\int_{0}^{a}(k+b)hh_bdb\nonumber\\
&=\varv(x,t,a) h(x,t,a)-\frac{1}{2}(k+a)h^2+\frac{1}{2}\int_{0}^{a}h^2(x,t,b)db,\nonumber
\end{align}
which gives (\ref{summarySLE1Dstream}). Let us prove (\ref{summarySLE1Dht}). First observe that $h(x,t,a)=\int_{0}^{a}h_bdb$, because $h(x,0,t)\equiv 0$. Thus, (\ref{summarySLE1Dht}) follows by integrating the second equation of (\ref{SLE1D}) in $b\in [0,a]$. Finally, by (\ref{summarySLE1Dht})

\begin{align}
\varv_t&=u_t(x,h,t)+(k+a)h_t\nonumber\\
&=-uu_x(x,h,t)-v(x,h,t)(k+a)-P_x(x,t)+(k+a)(v(x,h,t)-u(x,h,t)h_x)\nonumber\\
&=-uu_x(x,h,t)-P_x(x,t)-(k+a)u(x,h,t)h_x,\nonumber
\end{align}
which leads to (\ref{summarySLE1Dvt}) thanks to (\ref{summarySLE1Dhx}). This concludes the proof of Proposition \ref{summarySLE1D}. 

\end{proof}

\begin{proposition}\label{elementaryHSLE}
Let $(\vb{v},h_a)$ be a smooth solution to (\ref{HSLE}), then

\begin{align}
\partial_t\int_{\mathbb{T}^d\times (0,1)}\lvert \vb{v}\rvert^2h_adxda=0,\label{kineticHSLE}\\
P(x,t)=(-\Delta)^{-1}\left(\nabla\cdot\nabla\cdot\int_{0}^{1}(\vb{v}\otimes \vb{v} )h_bdb\right).\label{defpressureHSLE}
\end{align}

\end{proposition}

\begin{proof}
First observe that
\begin{align}
\partial_t\int_{0}^{1}h_ada&=-\nabla\cdot\int_{0}^{1}\vb{v}h_ada\nonumber\\
&=0.\label{divHSLE}
\end{align}

By (\ref{burgersHSLE}), (\ref{divHSLE}) and integration by parts
\begin{align}
\partial_t\int_{\mathbb{T}^d\times (0,1)}\lvert \vb{v}\rvert^2h_adxda&=\int_{\mathbb{T}^d\times (0,1)}2\vb{v}\cdot\vb{v_t}h_a+\lvert \vb{v}\rvert^2\partial_t h_adxda\nonumber\\
&=\int_{\mathbb{T}^d\times (0,1)}-\vb{v}\cdot\nabla\left(\lvert\vb{v}\rvert^2+2P\right)h_a-\nabla\cdot (\vb{v}h_a)\lvert\vb{v}\rvert^2dxda\nonumber\\
&=\int_{\mathbb{T}^d\times (0,1)}-\nabla\cdot (\vb{v}\lvert\vb{v}\rvert^2h_a)+2P\nabla\cdot(\vb{v}h_a) dxda\nonumber\\
&=2\int_{\mathbb{T}^d}P\nabla\cdot\left(\int_{0}^{1}\vb{v}h_ada\right) dxda\nonumber\\
&=0,\nonumber
\end{align}
which proves (\ref{kineticHSLE}). Now, let us prove (\ref{defpressureHSLE}). First observe that

\begin{align}
\partial_t\int_{0}^{1}\vb{v}h_ada&=-\int_{0}^{1}\left( \nabla\left(\frac{\lvert \vb{v}\rvert^2}{2}+P\right)h_a+\vb{v}\nabla\cdot(\vb{v}h_a)  \right)da \nonumber\\
&=-\nabla P-\nabla\cdot\int_{0}^{1}(\vb{v}\otimes \vb{v})h_ada,\nonumber
\end{align}
thanks to $\partial_i\varv_j=\partial_j\varv_i$. Applying divergence, we get:

$$0=-\Delta P-\nabla\cdot\nabla \cdot\int_{0}^{1}(\vb{v}\otimes \vb{v})h_ada, $$
which implies (\ref{defpressureHSLE}). This concludes the proof of Proposition \ref{elementaryHSLE}.

\end{proof}

\section{Monoticity and lower bounds for $E_1$ and $E_2$}\label{monoticitysection}

Throughout this section we will work in the framework of the local well-posedness result by Masmoudi and Wong \cite{MaTKW12}, namely, we will assume that $\omega_0\in H^4(\mathbb{T}\times (0,1))$ satisfies the local Rayleigh condition. We will show the monoticity of $E_1$ and $E_2$, as well as the validity of certain lower bounds that will be useful for proving Theorem \ref{collapse}.

The following two lemmas are elementary, so the proof will be omitted.
\begin{lemma}\label{materialderivative}

Assume that $\omega_0\in H^4(\mathbb{T}\times (0,1))$ satisfies the local Rayleigh condition. Let $\omega\in C([0,T];H^4(\mathbb{T}\times(0,1)))$ be the solution to (\ref{VHEE}) that satisfies the local Rayleigh condition in $[0,T]$. Let $f,g \in C^1([0,T];H^3(\mathbb{T}\times (0,1)))$. Denote by $D_t=\partial_t+u\partial_x+v\partial_y$ the material derivative, then:

\begin{align}
\partial_t\int_{\mathbb{T}\times(0,1)}f(x,y,t)dxdy &=\int_{\mathbb{T}\times(0,1)}D_tf(x,y,t)dxdy,\label{materialint}\\
D_t(fg) &=fD_tg+gD_tf,\label{materialproduct}\\
D_t\left(\frac{f}{g}\right) &=\frac{gD_tf-fD_tg}{g^2}, \text{\hspace{0.1cm} if g is strictly positive,}\label{materialfrac}\\
D_t\left(\log(g)\right) &=\frac{D_tg}{g}, \text{\hspace{0.1cm} if g is strictly positive,}\label{materiallog}
\end{align}

where $0\leq t \leq T$.

\end{lemma}

\begin{lemma}\label{lemmamaterialomega}
Assume that $\omega_0\in H^4(\mathbb{T}\times (0,1))$ satisfies the local Rayleigh condition. Let $\omega\in C([0,T];H^4(\mathbb{T}\times(0,1)))$ be the solution to (\ref{VHEE}) that satisfies the local Rayleigh condition in $[0,T]$. Denote by $D_t=\partial_t+u\partial_x+v\partial_y$ the material derivative, then:

\begin{align}
D_t (u)&=-P_x,\label{materialu}\\
D_t(\omega) &=0,\label{materialomega}\\
D_t (u_x) &=-u_x^2-v_x\omega-P_{xx},\label{materialux}\\
D_t (P_x) &=P_{xt}+uP_{xx},\label{materialpx}\\
D_t(\omega_x) &=-u_x\omega_x-v_x\omega_y,\label{materialomegax}\\
D_t(\omega_y) &=u_x\omega_y-\omega\omega_x,\label{materialomegay}
\end{align}
where $0\leq t\leq T$.

\end{lemma}

\begin{lemma}\label{lemmaformulas}

Assume that $\omega_0\in H^4(\mathbb{T}\times (0,1))$ satisfies the local Rayleigh condition. Let $\omega\in C([0,T];H^4(\mathbb{T}\times(0,1)))$ be the solution to (\ref{VHEE}) that satisfies the local Rayleigh condition in $[0,T]$. Denote by $D_t=\partial_t+u\partial_x+v\partial_y$ the material derivative, then:

\begin{align}
D_t\left(\log\left(\frac{1}{\omega_y}\right)\right)&=\frac{\omega\omega_x}{\omega_y}-u_x, \label{lemmaformulas1}\\
D_t\left(\frac{\omega\omega_x}{\omega_y}-u_x\right)&=\left\lvert\frac{\omega\omega_x}{\omega_y}-u_x\right\rvert^2+P_{xx},\label{lemmaformulas2}\\
D_t\left(u^2\left(\frac{\omega\omega_x}{\omega_y}-u_x\right)-uP_x\right)&=\left\lvert u\left(\frac{\omega\omega_x}{\omega_y}-u_x\right)-P_x \right\rvert^2-uP_{xt},\label{lemmaformulas3}
\end{align}

for $0\leq t<T$.

\end{lemma}

\begin{proof}
By Lemmas \ref{materialderivative} and \ref{lemmamaterialomega},

\begin{align}
D_t\left(\log\left(\frac{1}{\omega_y}\right)\right)&=-\frac{u_x\omega_y-\omega\omega_x}{\omega_y} \nonumber \\
&=\frac{\omega\omega_x}{\omega_y}-u_x, \nonumber 
\end{align}
which gives (\ref{lemmaformulas1}). Next, by Lemmas \ref{materialderivative} and \ref{lemmamaterialomega}, 
\begin{align}
D_t\left(\frac{\omega\omega_x}{\omega_y}\right)&=\frac{-u_x\omega\omega_x\omega_y-v_x\omega\omega_y-u_x\omega\omega_x\omega_y+\omega^2\omega_x^2}{\omega_y^2} \nonumber \\
&=\left\lvert \frac{\omega\omega_x}{\omega_y}-u_x  \right\rvert^2-\omega v_x-u_x^2. \label{materialomegaomegax}
\end{align}
Thus, by (\ref{materialomegaomegax}) and (\ref{materialux})

 $$D_t\left(\frac{\omega\omega_x}{\omega_y}-u_x \right)=\left\lvert \frac{\omega\omega_x}{\omega_y}-u_x  \right\rvert^2+P_{xx},$$
from which we get (\ref{lemmaformulas2}). Finally, by Lemmas \ref{materialderivative} and \ref{lemmamaterialomega} and (\ref{lemmaformulas2})
\begin{align}
& D_t\left(u^2\left(\frac{\omega\omega_x}{\omega_y}-u_x \right)-uP_x\right) \nonumber \\
 &=u^2\left(\left\lvert \frac{\omega\omega_x}{\omega_y}-u_x  \right\rvert^2+P_{xx}\right)-2uP_x\left(\frac{\omega\omega_x}{\omega_y}-u_x\right)+P_x^2-u^2P_{xx}-uP_{xt} \nonumber \\
&=\left\lvert u\left(\frac{\omega\omega_x}{\omega_y}-u_x\right)u-P_x  \right\rvert^2-uP_{xt},  \nonumber
\end{align}
which implies (\ref{lemmaformulas3}). This completes the proof of Lemma \ref{lemmaformulas}.
\end{proof}

\begin{proposition}\label{formulas}
Assume that $\omega_0\in H^4(\mathbb{T}\times (0,1))$ satisfies the local Rayleigh condition. Let $\omega\in C([0,T];H^4(\mathbb{T}\times(0,1)))$ be the solution to (\ref{VHEE}) that satisfies the local Rayleigh condition in $[0,T]$. Set

\begin{equation}\label{E1}
E_1(t)=\int_{\mathbb{T}\times (0,1)} \left(\frac{\omega\omega_x}{\omega_y}-u_x\right) dxdy =\int_{\mathbb{T}\times (0,1)} \frac{\omega\omega_x}{\omega_y} dxdy 
\end{equation}

and

\begin{equation}\label{E2}
 E_2(t)=\int_{\mathbb{T}\times (0,1)}\left(u^2 \left(\frac{\omega\omega_x}{\omega_y}-u_x\right)-uP_x\right) dxdy=\int_{\mathbb{T}\times (0,1)}u^2 \frac{\omega\omega_x}{\omega_y}dxdy.
\end{equation}
 
  Then 

\begin{align}
E_1 &=\partial_t \int_{\mathbb{T}\times (0,1)}\log\left(\frac{1}{\omega_y}\right)dxdy,\label{formulas1}\\
\partial_t E_1 &= \int_{\mathbb{T}\times (0,1)} \left\lvert u_x-\frac{\omega\omega_x}{\omega_y}\right\rvert^2dxdy,\label{formulas2}\\
\partial_t E_2 &= \int_{\mathbb{T}\times (0,1)} \left\lvert P_x+u\left(u_x-\frac{\omega\omega_x}{\omega_y}\right)\right\rvert^2dxdy,\label{formulas3}
\end{align}
where $0\leq t<T$.

\end{proposition}

\begin{proof}

By Lemma \ref{materialderivative}, (\ref{lemmaformulas1}) and the $x$-periodicity

\begin{equation}
\partial_t \int_{\mathbb{T}\times (0,1)}\log\left(\frac{1}{\omega_y}\right)dxdy=\int_{\mathbb{T}\times (0,1)}\left(\frac{\omega\omega_x}{\omega_y}-u_x\right)dxdy,
\end{equation}
which proves (\ref{formulas1}). Next, by Lemma \ref{materialderivative}, (\ref{lemmaformulas2}) and the $x$-periodicity
\begin{align}
\partial_t \int_{\mathbb{T}\times (0,1)}\left(\frac{\omega\omega_x}{\omega_y}-u_x\right)dxdy&=\int_{\mathbb{T}\times (0,1)}\left\lvert\frac{\omega\omega_x}{\omega_y}-u_x\right\rvert^2+P_{xx}dxdy \nonumber \\
&=\int_{\mathbb{T}\times (0,1)}\left\lvert\frac{\omega\omega_x}{\omega_y}-u_x\right\rvert^2dxdy, \nonumber
\end{align}
which gives (\ref{formulas2}). Finally, by Lemma \ref{materialderivative} and (\ref{lemmaformulas3})
\begin{align}
\partial_t E_2&=\int_{\mathbb{T}\times (0,1)} \left\lvert P_x+u\left(u_x-\frac{\omega\omega_x}{\omega_y}\right)\right\rvert^2-uP_{xt}dxdy \nonumber \\
&=\int_{\mathbb{T}\times (0,1)}\left\lvert P_x+u\left(u_x-\frac{\omega\omega_x}{\omega_y}\right)\right\rvert^2+u_xP_tdxdy,\nonumber
\end{align}
from which we get (\ref{formulas3}), thanks to (\ref{intux}). This concludes the proof of Proposition \ref{formulas}.

\end{proof}

\begin{lemma}\label{stationary1}
Assume that $\omega_0\in H^4(\mathbb{T}\times (0,1))$ satisfies the local Rayleigh condition. Let $\omega\in C([0,T];H^4(\mathbb{T}\times(0,1)))$ be the solution to (\ref{VHEE}) that satisfies the local Rayleigh condition in $[0,T]$. If $ \int_{\mathbb{T}\times (0,1)} \left\lvert u_x-\frac{\omega\omega_x}{\omega_y}\right\rvert^2dxdy=0$ (at any time) then $\omega$ is stationary and independent of $x$.

\end{lemma}

\begin{proof}

Assume that $\frac{\omega\omega_x}{\omega_y}-u_x\equiv 0$ at time $t=t_0$. Then,

\begin{equation}\label{stationary1derivative}
\partial_y\left(\frac{u_x}{\omega}\right)=\frac{\omega\omega_x-u_x\omega_y}{\omega^2}=\frac{\omega_y}{\omega^2}\left(\frac{\omega\omega_x}{\omega_y}-u_x\right)\equiv 0, \text{\hspace{0.1cm} if $\omega\neq 0$.}
\end{equation}
Let us first assume that $\omega$ does not vanishes. By (\ref{intux}), there exists $y_0\in[0,1]$ such that $u_x(x,y_0(x),t_0)=\frac{u_x(x,y_0(x),t_0)}{\omega(x,y_0(x),t_0)}=0$. By (\ref{stationary1derivative}), $u_x\equiv 0$ at time $t=t_0$. Thus, $\omega$ is $x$-independent at $t=t_0$. Moreover, $\omega(y,t_0)$ and $\omega(x,y,t_0+t)$ satisfy (\ref{VHEE}) with the same initial condition, since the solution is unique, $\omega(x,y,t_0+t)\equiv \omega(y,t_0)$ (see \cite{MaTKW12}).

Conversely, assume that $\omega$ vanishes in some points. Applying (\ref{stationary1derivative}), we have
 \begin{equation} \label{piecewiseu1}
    u_x(x,y,t_0) =
    \begin{cases}
      f(x,t_0)\omega(x,y,t_0)+c(x), & if \hspace{0.3cm} \omega>0, \\
      \tilde{f}(x,t_0)\omega(x,y,t_0)+\tilde{c}(x),        & if \hspace{0.3cm}  \omega<0, \\
      0, & if \hspace{0.3cm}  \omega=0. \\
     
    \end{cases}
  \end{equation}
Since $u_x$ is continuous, $c\equiv\tilde{c}\equiv 0$. Moreover, 

\begin{equation}
    \omega_x =
    \begin{cases}
      f(x,t_0)\omega_y, & if \hspace{0.3cm} \omega>0, \\
      \tilde{f}(x,t_0)\omega_y,        & if \hspace{0.3cm}  \omega<0, \\
      0, & if \hspace{0.3cm}  \omega=0. \\
     
    \end{cases}
  \end{equation}
Therefore, the continuity of $\frac{\omega_x}{\omega_y}$ yields $f(x,t_0)=\tilde{f}(x,t_0)$. Thus, $u_x=f(x,t_0)\omega$ and $\omega_x=f(x,t_0)\omega_y$. Integrating in $y\in [0,1]$, we get

\begin{equation}\label{stationary1ux}
\int_{0}^{1}u_xdy=0=f(x,t_0)\int_{0}^{1}\omega dy=f(x,t_0)u\rvert_{y=0}^{y=1},
\end{equation}

\begin{equation}\label{stationary1omegax}
\int_{0}^{1}\omega_xdy=f(x,t_0)\int_{0}^{1}\omega_y dy=\partial_x u\rvert_{y=0}^{y=1}.
\end{equation}
Applying (\ref{stationary1ux}), (\ref{stationary1omegax}) and the positivity of $\omega_y$, we get

\begin{equation}\label{stationary1final}
\partial_x \left(\left\lvert u(1)-u(0)\right\rvert^2\right)=2u\rvert_{y=0}^{y=1}\partial_x u\rvert_{y=0}^{y=1}\equiv 0.
\end{equation}
Finally, if $f(x_0,t_0)\neq 0$, (\ref{stationary1ux}) implies $u(x_0,y,t_0)\rvert_{y=0}^{y=1}=0$. Moreover, by (\ref{stationary1final}), $u(x,y,t_0)\rvert_{y=0}^{y=1}\equiv 0$, which contradicts  (\ref{stationary1omegax}).
Thus, $f\equiv u_x\equiv 0$ at time $t=t_0$. Furthermore, $\omega(y,t_0)$ and $\omega(x,y,t_0+t)$ satisfy (\ref{VHEE}) with the same initial data. Then, by uniqueness $\omega(x,y,t_0+t)\equiv \omega(y,t_0)$ (see \cite{MaTKW12}). This completes the proof of Lemma \ref{stationary1}.

\end{proof}

\begin{corollary}\label{stationary}
Let $\omega\in H^4(\mathbb{T}\times (0,1))$ be a stationary solution to (\ref{VHEE}) that satisfies the local Rayleigh condition. Then, $\omega$ is $x$-independent. Conversely, if $\omega$ is an $x$-independent solution to (\ref{VHEE}), then $\omega$ is stationary.

\end{corollary}

\begin{proof}

Let $\omega$ be a stationary solution, then by (\ref{formulas2}):

$$\partial_t E_1=\int_{\mathbb{T}\times (0,1)}\left\lvert  \frac{\omega\omega_x}{\omega_y}-u_x\right\rvert^2dxdy=0.$$By Lemma \ref{stationary1}, $\omega$ is $x$-independent.
On the other hand, if $\omega$ is an $x$-independent solution to (\ref{VHEE}), then

$$ \omega_t=-u\omega_x-\mathcal{A}(\omega_x)\omega_y\equiv 0,$$
where $\mathcal{A}(\omega)$ is defined by (\ref{eqstream}), which concludes the proof of Corollary \ref{stationary}.

\end{proof}

\begin{lemma}\label{stationary2}
Assume that $\omega_0\in H^4(\mathbb{T}\times (0,1))$ satisfies the local Rayleigh condition. Let $\omega\in C([0,T];H^4(\mathbb{T}\times(0,1)))$ be the solution to (\ref{VHEE}) that satisfies the local Rayleigh condition in $[0,T]$. If $ \int_{\mathbb{T}\times (0,1)} \left\lvert P_x+u\left(u_x-\frac{\omega\omega_x}{\omega_y}\right)\right\rvert^2dxdy=0$ (at any time) then $\omega$ is stationary and independent of $x$.

\end{lemma}

\begin{proof}

Assume that $P_x+u\left(u_x-\frac{\omega\omega_x}{\omega_y}\right)\equiv 0$ at time $t=t_0$. Then,

\begin{equation}\label{stationary2prop}
P_x+u\left(u_x-\frac{\omega\omega_x}{\omega_y}\right)= P_x+uu_x+\frac{v\omega\omega_y}{\omega_y}-\frac{\omega v\omega_y}{\omega_y}-\frac{\omega u\omega_x}{\omega_y}=\frac{\omega\omega_t}{\omega_y}-u_t.
\end{equation}
Furthermore,

\begin{equation}\label{stationary2derivative}
\partial_y\left(\frac{u_t}{\omega}\right)=\frac{\omega\omega_t-u_t\omega_y}{\omega^2}=\frac{\omega_y}{\omega^2}\left(\frac{\omega\omega_t}{\omega_y}-u_t\right)\equiv 0, \text{\hspace{0.1cm} if $\omega\neq 0$.}
\end{equation}
Let us first assume that $\omega$ does not vanishes. By (\ref{intut}), there exists $y_0\in[0,1]$ such that $u_t(x,y_0(x),t_0)=\frac{u_t(x,y_0(x),t_0)}{\omega(x,y_0(x),t_0)}=0$. By (\ref{stationary2derivative}), $u_t\equiv 0$ at time $t=t_0$. Moreover, $\omega(x,y,t_0)$ and $\omega(x,y,t_0+t)$ satisfy (\ref{VHEE}) with the same initial condition. Since the solution is unique, $\omega(x,y,t_0+t)\equiv \omega(x,y,t_0)$. Furthermore, by Corollary \ref{stationary}, $\omega$ is $x$-independent.

Conversely, assume that $\omega$ vanishes in some points. Applying (\ref{stationary2derivative}), we have
 \begin{equation} \label{piecewiseu2}
    u_t(x,y,t_0) =
    \begin{cases}
      f(x,t_0)\omega(x,y,t_0)+c(x), & if \hspace{0.3cm} \omega>0, \\
      \tilde{f}(x,t_0)\omega(x,y,t_0)+\tilde{c}(x),        & if \hspace{0.3cm}  \omega<0, \\
      0, & if \hspace{0.3cm}  \omega=0. \\
     
    \end{cases}
  \end{equation}
Since $u_t$ is continuous, $c\equiv\tilde{c}\equiv 0$. Moreover, 

\begin{equation}
    \omega_t =
    \begin{cases}
      f(x,t_0)\omega_y, & if \hspace{0.3cm} \omega>0, \\
      \tilde{f}(x,t_0)\omega_y,        & if \hspace{0.3cm}  \omega<0, \\
      0, & if \hspace{0.3cm}  \omega=0. \\
     
    \end{cases}
  \end{equation}
Therefore, the continuity of $\frac{\omega_t}{\omega_y}$ yields $f(x,t_0)\equiv\tilde{f}(x,t_0)$. Thus, $u_t=f(x,t_0)\omega$ and $\omega_t=f(x,t_0)\omega_y$. Integrating in $y\in [0,1]$, we get

\begin{equation}\label{stationary2ut}
\int_{0}^{1}u_tdy=0=f(x,t_0)\int_{0}^{1}\omega dy=f(x,t_0)u\rvert_{y=0}^{y=1},
\end{equation}

\begin{equation}\label{stationary2omegat}
\int_{0}^{1}\omega_tdy=f(x,t_0)\int_{0}^{1}\omega_y dy=\partial_t u\rvert_{y=0}^{y=1}=\frac{1}{2}\partial_x(u^2(0)-u^2(1)).
\end{equation}
Applying (\ref{stationary2ut}), (\ref{stationary2omegat}) and the positivity of $\omega_y$, we get

\begin{equation}\label{stationary2final}
\partial_t \left(\left\lvert u(1)-u(0)\right\rvert^2\right)=2u\rvert_{y=0}^{y=1}\partial_t u\rvert_{y=0}^{y=1}=2u\rvert_{y=0}^{y=1}\partial_x\left(u^2(0)-u^2(1)\right)\equiv 0.
\end{equation}
Multiplying by $u(1)+u(0)$, we obtain

\begin{equation}\label{stationary22final}
\partial_x\left(\left\lvert u^2(0)-u^2(1)\right\rvert^2\right)\equiv 0.
\end{equation}
Finally, if $f(x_0,t_0)\neq 0$, (\ref{stationary2ut}) implies $u(x_0,y,t_0)\rvert_{y=0}^{y=1}=0$. Moreover, by (\ref{stationary22final}) and (\ref{stationary2omegat}), $u(x,y,t_0)\rvert_{y=0}^{y=1}\equiv 0$, which contradicts  (\ref{stationary2omegat}).
Thus, $f\equiv u_t\equiv 0$ at time $t=t_0$. Furthermore, $\omega(x,y,t_0)$ and $\omega(x,y,t_0+t)$ satisfy (\ref{VHEE}) with the same initial data. Then, by uniqueness $\omega(x,y,t_0+t)\equiv \omega(x,y,t_0)$ (see \cite{MaTKW12}). Moreover, since $\omega$ is stationary, $\partial_t E_1=0$. Applying Lemma \ref{stationary1}, we deduce that $\omega$ is $x$-independent. This concludes the proof of Lemma \ref{stationary2}.

\end{proof}

\begin{lemma}\label{blowuplemma}

Let $f:[0,L)\rightarrow \mathbb{R}$ be a $C^1$ function and $C>0$. Assume that $\partial_t f\geq Cf^2$ and $f(0)>0$. Then there exists a positive time $T^{*}$ such that $f\rightarrow \infty$ as $t\rightarrow T^{*}$. Moreover, $T^{*}\leq \frac{1}{Cf(0)}$ and $f$ satisfies:

\begin{equation}
f(t)\geq \frac{1}{C}\frac{1}{\frac{1}{Cf(0)}-t}, 
\end{equation}

 for $t<T^*$.

\end{lemma}

\begin{proof}

Since $f$ is positive we have:

$$\frac{\partial_s f}{f^2}\geq C.$$
Integrating the above we get:

$$\int_{0}^{t}\frac{f_s}{f^2}ds=\frac{1}{f(0)}-\frac{1}{f(t)}\geq C\int_{0}^{t}ds=Ct.$$
From which we obtain:

$$f(t)\geq \frac{1}{\frac{1}{f(0)}-Ct}, $$
which concludes the proof.
\end{proof}

\begin{proposition}\label{blowupE1}

Assume that $\omega_0\in H^4(\mathbb{T}\times (0,1))$ satisfies the local Rayleigh condition and $\partial_x\omega_0$ is not identically $0$. Let $\omega\in C([0,T];H^4(\mathbb{T}\times(0,1)))$ be the solution to (\ref{VHEE}) that satisfies the local Rayleigh condition in $[0,T]$. Then, $E_1$ satisfies:

\begin{equation}\label{blowupE1ineq}
E_1(t)\geq \frac{1}{\frac{1}{E_1(0)}-t}, \text{\hspace{1cm} if $E_1(0)>0,$}
\end{equation}
where $0\leq t\leq  T$.

Furthermore, if $E_1(0)=0$ there exists a positive time $t'$ (sufficiently small) such that $E_1(t')>0$ and 
\begin{equation}
E_1(t+t')\geq \frac{1}{\frac{1}{E_1(t')}-t}, 
\end{equation}

where $0\leq t\leq  T$.

\end{proposition}

\begin{proof}

By (\ref{formulas2}) and H\"older's inequality,

\begin{align}
\partial_t E_1 &= \int_{\mathbb{T}\times (0,1)} \left\lvert u_x-\frac{\omega\omega_x}{\omega_y}\right\rvert^2dxdy \nonumber \\
&\geq \left(  \int_{\mathbb{T}\times (0,1)} \left(u_x-\frac{\omega\omega_x}{\omega_y}\right)dxdy\right)^2 \nonumber \\
&=E_1^2(t). \label{quadraticgrowthE1}
\end{align}
If $E_1(0)>0$ the result follows by applying Lemma \ref{blowuplemma}. On the other hand, if $E_1(0)=0$, Lemma \ref{stationary1} implies

$$\int_{\mathbb{T}\times (0,1)}\left\lvert\frac{\omega\omega_x}{\omega_y}-u_x\right\rvert^2dxdy>0,$$
at every time. Therefore, for every positive time $t'$ we have
$$\int_{0}^{t'}\int_{\mathbb{T}\times (0,1)}\left\lvert\frac{\omega\omega_x}{\omega_y}-u_x\right\rvert^2dxdydt=\int_{0}^{t'}\partial_tE_1(t)dt=E_1(t')-E_1(0)=E_1(t')>0.$$
Choose $t'$ such that $\omega\in C([0,T+t'];H^4(\mathbb{T}\times(0,1)))$ satisfies the local Rayleigh condition in $[0,T+t']$. Applying (\ref{blowupE1ineq}) to $\omega(t+t')$ we get

$$ E_1(t'+t)\geq \frac{1}{\frac{1}{E_1(t')}-t},$$
for $0\leq t\leq T$, which completes the proof of Proposition \ref{blowupE1}.
\end{proof}

\begin{proposition}\label{blowupE2}

Assume that $\omega_0\in H^4(\mathbb{T}\times (0,1))$ satisfies the local Rayleigh condition and $\partial_x\omega_0$ is not identically $0$. Let $\omega\in C([0,T];H^4(\mathbb{T}\times(0,1)))$ be the solution to (\ref{VHEE}) that satisfies the local Rayleigh condition in $[0,T]$. Then, $E_2$ satisfies:

\begin{equation}\label{blowupE2ineq}
E_2(t)\geq \frac{\Vert u\Vert_{2}^{2}}{\frac{\Vert u\Vert_{2}^{2}}{E_2(0)}-t}, \text{\hspace{1cm} if $E_2(0)>0,$}
\end{equation}
where $0\leq t\leq T$. Furthermore, if $E_2(0)=0$ there exists a positive time $t'$ (sufficiently small) such that $E_2(t')>0$ and 
\begin{equation}
E_2(t'+t)\geq \frac{\Vert u\Vert_{2}^{2}}{\frac{\Vert u\Vert_{2}^{2}}{E_2(t')}-t}, 
\end{equation}
where $0\leq t\leq T$.

\end{proposition}

\begin{proof}

By (\ref{formulas3}), H\"older's inequality, (\ref{intux}) and the $x$-periodicity:

\begin{align}
&\Vert u\Vert_{2}^2\cdot\partial_t E_2(t) \nonumber \\
&=\left(\int_{\mathbb{T}\times (0,1)}u^2dxdy\right)\cdot \int_{\mathbb{T}\times (0,1)} \left\lvert P_x+u\left(u_x-\frac{\omega\omega_x}{\omega_y}\right)\right\rvert^2dxdy \nonumber \\
&\geq \left(\int_{\mathbb{T}\times (0,1)}u\left( P_x+u\left(u_x-\frac{\omega\omega_x}{\omega_y} \right)\right)dxdy\right)^2 \nonumber \\
&= \left(\int_{\mathbb{T}}\left(\int_{0}^{1}udy\right)P_xdx+\int_{\mathbb{T}\times (0,1)}\partial_x\left(\frac{u^3}{3}\right)dxdy-\int_{\mathbb{T}\times (0,1)}u^2\frac{\omega\omega_x}{\omega_y}dxdy\right)^2 \nonumber \\
&= \left(\int_{\mathbb{T}\times (0,1)}u^2\frac{\omega\omega_x}{\omega_y}dxdy \right)^2 \nonumber \\
&=E_2^2(t).\label{quadraticgrowthE2}
\end{align}
Thus $$ \partial_t E_2\geq \Vert u\Vert_{2}^{-2}E_2^2, $$
where $\Vert u\Vert_{2}^2$ is time independent thanks to (\ref{kinetic}). If $E_2(0)>0$, Lemma \ref{blowuplemma} yields

$$ E_2(t)\geq \frac{\Vert u\Vert_{2}^{2}}{\frac{\Vert u\Vert_{2}^{2}}{E_2(0)}-t},$$
where $0\leq t\leq T$. On the other hand, if $E_2(0)=0$, Lemma \ref{stationary2} implies

$$\int_{\mathbb{T}\times (0,1)} \left\lvert P_x+u\left(u_x-\frac{\omega\omega_x}{\omega_y}\right)\right\rvert^2dxdy>0,$$
at every time, from which we get

$$ \int_{0}^{t'}\int_{\mathbb{T}\times (0,1)} \left\lvert P_x+u\left(u_x-\frac{\omega\omega_x}{\omega_y}\right)\right\rvert^2dxdyd\tau=\int_{0}^{t'}\partial_{\tau} E_2d\tau=E_2(t')>0,$$
for every $t'>0$. Choose $t'$ such that $\omega\in C([0,T+t'];H^4(\mathbb{T}\times(0,1)))$ satisfies the local Rayleigh condition in $[0,T+t']$. Applying (\ref{blowupE2ineq}) to $\omega(t+t')$ we get

$$ E_2(t'+t)\geq \frac{\Vert u\Vert_{2}^{2}}{\frac{\Vert u\Vert_{2}^{2}}{E_2(t')}-t},$$
for $0\leq t\leq T$, which concludes the proof of Proposition \ref{blowupE1}.

\end{proof}

\section{Proof of Theorem \ref{collapse}: Collapse of the local rayleigh condition or singularity formation}\label{proofcollapse}

\begin{proof}
Let us proceed by contradiction. Suppose that  $\Vert \omega (t)\Vert_{H^{4}(\mathbb{T}\times(0,1))}$ and $\left\Vert\frac{1}{\omega_y (t)} \right\Vert_{L^{\infty}(\mathbb{T}\times(0,1))}$ remain bounded for every time $t>0$. Assume first that $E_1(0)\geq 0$. By Proposition \ref{blowupE1}, $E_1$ blows up in finite time. Moreover, by (\ref{formulas2}) and Poincare's inequality
\begin{align}
& E_1(\tau)\rvert_{\tau=0}^{\tau=t} \nonumber \\
&=\int_{0}^{t}\int_{\mathbb{T}\times (0,1)} \left\lvert u_x-\frac{\omega\omega_x}{\omega_y}\right\rvert^2dxdyd\tau\nonumber \\
&\leq 2\int_{0}^{t}\int_{\mathbb{T}\times (0,1)}u_x^2+\left(\frac{\omega\omega_x}{\omega_y}\right)^2dxdyd\tau \nonumber \\
& \leq 2\int_{0}^{t}\int_{\mathbb{T}\times (0,1)}\left(\frac{\omega_x^2}{\pi^2}+\Vert\omega_0\Vert_{\infty}^2\left(\frac{\omega_x}{\omega_y}\right)^2\right)dxdyd\tau \nonumber \\
& \leq \left(\frac{2}{\pi^2}+2\Vert\omega_0\Vert_{\infty}^2\right)\int_{0}^{t}\int_{\mathbb{T}\times (0,1)}\omega_x^2\left(1+\frac{1}{\omega_y^2}\right)dxdyd\tau \label{collapseE12}\\
& \leq \left(\frac{2}{\pi^2}+2\Vert\omega_0\Vert_{\infty}^2\right)\int_{0}^{t}\Vert \omega_x (\tau)\Vert_{L^{2}(\mathbb{T}\times(0,1))}^2\left(1+\left\Vert\frac{1}{\omega_y (\tau)} \right\Vert_{L^{\infty}(\mathbb{T}\times(0,1))}^2\right)d\tau , \nonumber
\end{align}
which yields a contradiction because $t$ is arbritrary and $E_1$ blows up in finite time. Next, if $E_2(0)\geq 0$, by Proposition \ref{blowupE2}, $E_2$ blows up in finite time. Furthermore, by (\ref{formulas3})

\begin{align}
E_2(\tau)\rvert_{\tau=0}^{\tau=t} &=\int_{0}^{t}\int_{\mathbb{T}\times (0,1)} \left\lvert P_x+u\left(u_x-\frac{\omega\omega_x}{\omega_y}\right)\right\rvert^2dxdyd\tau \nonumber \\
&\leq 2\int_{0}^{t}\int_{\mathbb{T}\times (0,1)}\left(P_x^2+ u^2\left\lvert u_x-\frac{\omega\omega_x}{\omega_y}\right\rvert^2\right)dxdyd\tau. \nonumber
\end{align}
By (\ref{defpressure}), H\"older's inequality, Poincare's inequality and (\ref{uinfty})

\begin{align}
\int_{\mathbb{T}}P_x^2dx &=4\int_{\mathbb{T}}\left(\int_{0}^{1}uu_xdy\right)^2dx\nonumber \\
&\leq 4\int_{\mathbb{T}}\int_{0}^{1}u^2dy\int_{0}^{1}u_x^2dydx \nonumber \\
&\leq 4\left(\frac{3}{2}\Vert \omega_0\Vert_{\infty}\right)^2\frac{1}{\pi^2}\int_{\mathbb{T}\times (0,1)}\omega_x^2dxdy\nonumber \\
&=\left(\frac{3}{\pi}\Vert\omega_0\Vert_{\infty}\right)^2\int_{\mathbb{T}\times (0,1)}\omega_x^2dxdy. \nonumber
\end{align}
Applying (\ref{collapseE12}) and (\ref{uinfty}), we get
\begin{align}
& E_2(\tau)\rvert_{\tau=0}^{\tau=t} \nonumber \\
&\leq \left(2\left(\frac{3}{\pi}\Vert\omega\Vert_{\infty}\right)^2+2\left(\frac{2}{\pi^2}+2\Vert\omega_0\Vert_{\infty}^2\right)\right)\int_{0}^{t}\int_{\mathbb{T}\times (0,1)}\omega_x^2\left(1+\frac{1}{\omega_y^2}\right)dxdyd\tau \label{collapseE22} \\
&\leq   \left( 6\Vert\omega_0\Vert_{\infty}+\frac{4}{\pi^2}\right)\int_{0}^{t}\Vert \omega_x (\tau)\Vert_{L^{2}(\mathbb{T}\times(0,1))}^2\left(1+\left\Vert\frac{1}{\omega_y (\tau)} \right\Vert_{L^{\infty}(\mathbb{T}\times(0,1))}^2\right)d\tau  , \nonumber
\end{align}
which again yields a contradiction because $t$ is arbritrary and $E_2$ blows up in finite time. From now on, we will assume that $\Vert \omega (t)\Vert_{H^{4}(\mathbb{T}\times(0,1))}$ and $\Vert\frac{1}{\omega_y (t)} \Vert_{L^{\infty}(\mathbb{T}\times(0,1))}$ remain bounded, so we can apply the local well-posedness result by Masmoudi and Wong \cite{MaTKW12}. First let us prove (\ref{collapselogE1}). Assume that $E_1(0)>0$. By Proposition \ref{blowupE1},

$$\frac{1}{\frac{1}{E_1(0)}-\tau}\leq E_1(\tau).$$
Applying (\ref{formulas1}) and integrating in $\tau\in[0,t]$, we get:
\begin{align}
\log\left(\frac{1}{E_1(0)}\frac{1}{\frac{1}{E_1(0)}-t}\right) & =\int_{0}^{t}\frac{1}{\frac{1}{E_1(0)}-\tau}d\tau \nonumber \\
&\leq \int_{0}^{t}E_1(\tau)d\tau \nonumber \\
& =\int_{0}^{t}\partial_{\tau}\int_{\mathbb{T}\times (0,1)}\log\left(\frac{1}{\omega_y(\tau)}\right)dxdyd\tau \nonumber \\
&=\int_{\mathbb{T}\times (0,1)}\log\left(\frac{\partial_y\omega_0}{\omega_y(t)}\right)dxdy,\nonumber
\end{align}
which completes the proof of (\ref{collapselogE1}). Let us now deduce the estimate (\ref{collapselogE2}). Assume that $E_2(0)>0$. By Proposition \ref{blowupE2}, 

\begin{equation}\label{collapseE2blowupE2}
\left(\frac{\Vert u\Vert_{2}^{2}}{\frac{\Vert u\Vert_{2}^{2}}{E_2(0)}-t}\right)^2\leq E_2^2(t).
\end{equation}
Moreover, by H\"older's inequality and (\ref{formulas2}),

\begin{align}
E_2^2(t) &=\left( \int_{\mathbb{T}\times(0,1)}u^2\left(\frac{\omega\omega_x}{\omega_y}-u_x\right)dxdy\right)^2\nonumber \\
&\leq \Vert u\Vert_4^4\int_{\mathbb{T}\times(0,1)}\left\lvert \frac{\omega\omega_x}{\omega_y}-u_x  \right\rvert^2dxdy \nonumber \\
&=\Vert u\Vert_4^4\partial_t E_1. \label{collapseE2blowupE22}
\end{align}
Therefore, by (\ref{collapseE2blowupE2}), (\ref{collapseE2blowupE22}) and (\ref{uinfty}),

\begin{align}
\frac{\Vert u\Vert_{2}^{4}}{\frac{\Vert u\Vert_{2}^{2}}{E_2(0)}-t}-E_2(0)\Vert u\Vert_{2}^{2}& =\int_{0}^{t}  \left(\frac{\Vert u\Vert_{2}^{2}}{\frac{\Vert u\Vert_{2}^{2}}{E_2(0)}-\tau}\right)^2 d\tau \nonumber \\
&\leq \int_{0}^{t}  E_2^2(\tau)d\tau \nonumber \\
&\leq \int_{0}^{t}\Vert u\Vert_{4}^{4}\partial_{\tau}E_1(\tau)d\tau \nonumber \\
&\leq \left(\frac{3}{2}\Vert\omega_0\Vert_{\infty}\right)^4(E_1(t)-E_1(0)) \nonumber \\
&\leq \left(\frac{3}{2}\Vert\omega_0\Vert_{\infty}\right)^4(E_1(t)+\lvert E_1(0)\rvert). \nonumber
\end{align}
Thus

$$ \frac{1}{\frac{\Vert u\Vert_{2}^{2}}{E_2(0)}-t}\leq \left(\frac{E_2(0)}{\Vert u\Vert_2^2}+\left( \frac{3}{2}\frac{\Vert\omega_0\Vert_{\infty}}{\Vert u\Vert_{2}}\right)^4\lvert E_1(0)\rvert\right)+\left( \frac{3}{2}\frac{\Vert\omega_0\Vert_{\infty}}{\Vert u\Vert_{2}}\right)^4E_1(t).  $$
Integrating in time, we get

$$\log \left(\frac{\Vert u\Vert_2^2}{E_2(0)}\frac{1}{\frac{\Vert u\Vert_2^2}{E_2(0)}-t}\right)\leq tC_1(\omega_0)+ C_2(\omega_0)\int_{\mathbb{T}\times (0,1)}\log\left(\frac{\partial_y\omega_0}{\omega_y(t)}\right)dxdy, $$
which concludes the proof of (\ref{collapselogE2}). Finally, Propositions \ref{blowupE1} and \ref{blowupE2}, (\ref{collapseE12})  and (\ref{collapseE22}) lead to (\ref{collapseE1}) and (\ref{collapseE2}). This completes the proof of Theorem \ref{collapse}.

\end{proof}

\section{Proof of Theorem \ref{properties}: Necessary conditions for global solvability}\label{proofproperties}

In this section we will prove Theorem \ref{properties}, which establishes necessary conditions for the global solvability of (\ref{HEE}). The following proposition will be useful for proving Theorem \ref{properties}.

\begin{proposition}\label{pxl2}

Assume that $\omega_0\in H^4(\mathbb{T}\times (0,1))$ satisfies the local Rayleigh condition. Let $\omega\in C([0,T];H^4(\mathbb{T}\times(0,1)))$ be the solution to (\ref{VHEE}) that satisfies the local Rayleigh condition in $[0,T]$. Let $E_1$ and $E_2$ be defined by (\ref{defE1}) and (\ref{defE2}) respectively. Then, we have:

\begin{equation}\label{pxl2eq}
\int_{0}^{T}\int_{\mathbb{T}}P_x^2dxdt\leq 2E_2\rvert_{t=0}^{t=T}+\frac{9}{2}\Vert \omega_0\Vert_{\infty}^2E_1\rvert_{t=0}^{t=T}.
\end{equation}

\end{proposition}

\begin{proof}
By (\ref{formulas3}), (\ref{formulas2}) and (\ref{uinfty})

\begin{align}
 \int_{0}^{T}\int_{\mathbb{T}}P_x^2dxdt&=\int_{0}^{T}\int_{\mathbb{T}\times (0,1)}\left\lvert P_x +u\left(u_x-\frac{\omega\omega_x}{\omega_y}\right)+u\left(\frac{\omega\omega_x}{\omega_y}-u_x\right)\right\rvert^2dxdydt \nonumber \\
 &\leq 2\int_{0}^{T}\int_{\mathbb{T}\times (0,1)}\left\lvert P_x +u\left(u_x-\frac{\omega\omega_x}{\omega_y}\right)\right\rvert^2+\left\lvert u\left(\frac{\omega\omega_x}{\omega_y}-u_x\right)\right\rvert^2dxdydt \nonumber \\
 &\leq 2\int_{0}^{T}\partial_t E_2(t)dt+2\left(\frac{3}{2}\Vert\omega_0\Vert_{\infty}\right)^2\int_{0}^{T}\int_{\mathbb{T}\times (0,1)}\left\lvert \frac{\omega\omega_x}{\omega_y}-u_x\right\rvert^2dxdydt \nonumber \\
 &=2\int_{0}^{T}\partial_t E_2(t)dt+2\left(\frac{3}{2}\Vert\omega_0\Vert_{\infty}\right)^2\int_{0}^{T}\partial_t E_1(t)dt \nonumber \\
 &=2E_2\rvert_{t=0}^{t=T}+\frac{9}{2}\Vert\omega_0\Vert_{\infty}E_1\rvert_{t=0}^{t=T},\nonumber
 \end{align}
which concludes the proof of Proposition \ref{pxl2}.

\end{proof}

Now we are ready to prove Theorem \ref{properties}.

\begin{proof}[Proof of Theorem \ref{properties}]

By Theorem \ref{collapse}, $E_1(t)$ and $E_2(t)$ remain negative, for all $t>0$. Moreover, by (\ref{formulas2})and  (\ref{formulas3}), $E_1(t)$ and $E_2(t)$ are convergent (since they are monotone and bounded). Furthermore, by (\ref{quadraticgrowthE1}) and (\ref{quadraticgrowthE2}),

\begin{equation}\label{intgrowthE1}
\int_{0}^{t}E_1^2d\tau\leq E_1\rvert_{\tau=0}^{\tau=t}.
\end{equation}

\begin{equation}\label{intgrowthE2}
\int_{0}^{t}E_2^2d\tau\leq \Vert u\Vert_{2}^2E_2\rvert_{\tau=0}^{\tau=t}.
\end{equation}
Which implies that $E_1(t),E_2(t)\rightarrow 0$, as $t\rightarrow \infty$ (otherwise (\ref{intgrowthE1}) and (\ref{intgrowthE2}) would growth arbitrarily, which is a contradiction). Thus,

\begin{equation}\label{intgrowthE11}
E_1\rvert_{\tau=0}^{\tau=t}= \int_{0}^{t}\int_{\mathbb{T}\times (0,1)} \left\lvert \frac{\omega\omega_x}{\omega_y}-u_x\right\rvert^2dxdyd\tau\rightarrow -E_1(0), \text{\hspace{0.1cm} as $t\rightarrow \infty$}
\end{equation}

and
\begin{equation}\label{intgrowthE22}
E_2\rvert_{\tau=0}^{\tau=t}= \int_{0}^{t}\int_{\mathbb{T}\times (0,1)} \left\lvert P_x+u\left(u_x-\frac{\omega\omega_x}{\omega_y}\right)\right\rvert^2dxdyd\tau\rightarrow -E_2(0), \text{\hspace{0.1cm} as $t\rightarrow \infty$},
\end{equation}
which proves (\ref{properties1}) and (\ref{properties2}). Next, Proposition \ref{pxl2}, (\ref{intgrowthE11}) and (\ref{intgrowthE22}) lead to (\ref{properties3}). Finally, by Jensen's inequality

\begin{align}
\exp\left(\int_{\mathbb{T}\times (0,1)}\log\left(\omega_y(t)\right)dxdy\right)&\leq \int_{\mathbb{T}\times (0,1)}\omega_ydxdy \nonumber \\
&\leq 2\Vert\omega_0\Vert_{\infty}.\label{propertiesjensen}
\end{align}
Thus, (\ref{formulas1}) and (\ref{propertiesjensen}) implies

\begin{align}
\int_{0}^{T}\lvert E_1(t)\rvert dt&=-\int_{0}^{T}E_1(t)dt \nonumber \\
&=\int_{\mathbb{T}\times (0,1)}\log\left(\frac{\omega_y(T)}{\partial_y\omega_0}\right)dxdy \nonumber \\
&\leq \int_{\mathbb{T}\times (0,1)}\log\left(\frac{2\Vert\omega_0\Vert_{\infty}}{\partial_y\omega_0}\right)dxdy,\nonumber
\end{align}
which gives (\ref{properties4}) because $T$ is arbitrary. This completes the proof of Theorem \ref{properties}.

\end{proof}

\section{Propagation of the local Rayleigh condition and $H^1$ control}\label{propagation}

It is worth mentioning that Theorem \ref{collapse} does not give any information about what happens first; the collapse of the local Rayleigh condition or the formation of singularities, even if the lower bounds (\ref{collapselogE1}) and (\ref{collapselogE2}) hold, Theorem \ref{collapse} does not guarantees that $\Vert\omega\Vert_{H^4(\mathbb{T}\times (0,1))}$ remains bounded before the collapse of the local Rayleigh condition. Nevertheless, if we assume that the vorticity $\omega$ does not vanishes, then we can control the $L^{\infty}([0,T];H^1(\mathbb{T}\times (0,1)))$ norm of $\omega$ and the $L^{\infty}([0,T];L^{\infty}(\mathbb{T}\times (0,1)))$   norm of $\frac{1}{\omega_y}$ by means of  
\begin{equation}
M(T)=\int_{0}^{T}\left\lvert\frac{\omega\omega_x}{\omega_y}-u_x \right\rvert_{\infty}d\tau,
\end{equation}
which is the main result of this section. The next proposition will be useful for proving Theorem \ref{continuationcriteria}.

\begin{proposition}\label{omegaxl2prop}
Assume that $\omega_0\in H^4(\mathbb{T}\times (0,1))$ satisfies the local Rayleigh condition. Let $\omega\in C([0,T];H^4(\mathbb{T}\times(0,1)))$ be the solution to (\ref{VHEE}) that satisfies the local Rayleigh condition in $[0,T]$. Assume further that $\lvert\omega_0\rvert >0$ in $\mathbb{T}\times(0,1)$. Then, we have:
\begin{align}
\lvert u_x\rvert_{\infty}&\leq \left(\frac{\Vert \omega_0 \Vert_{\infty}}{\min \lvert\omega_0\rvert}+1\right)\left\lvert \frac{\omega\omega_x}{\omega_y}-u_x\right\rvert_{\infty},\label{uxinfty} \\
\int_{\mathbb{T}\times(0,1)}\frac{\omega_{x}^2}{\omega_y}dxdy\rvert_{t=T}&\leq \int_{\mathbb{T}\times(0,1)}\frac{\lvert\partial_x\omega_{0\rvert}^2}{\partial_y\omega_0}dxdy\exp\left(C(\omega_0)\int_{0}^{T}\left\lvert\frac{\omega\omega_x}{\omega_y}-u_x\right\rvert_{\infty}dt \right),\label{omegaxl2bound} 
\end{align}
where $C(\omega_0)=3+\frac{2\Vert \omega_0 \Vert_{\infty}}{\min \lvert\omega_0\rvert}$.
\end{proposition}

\begin{proof}

Let $x$ be any point in $\mathbb{T}$. By (\ref{intux}), there exists $y_0(x)\in[0,1]$ such that  $u_x(x,y_0(x),t)=0$. The positivity of $\lvert \omega\rvert$ yields
\begin{align}
\frac{u_x}{\omega}(x,\tilde{y},t)&=\int_{y_0(x)}^{\tilde{y}}\frac{\omega\omega_x-u_x\omega_y}{\omega^2}dy \nonumber \\ &=\int_{y_0(x)}^{\tilde{y}}\frac{\omega\omega_x-u_x\omega_y}{\omega_y}\frac{\omega_y}{\omega^2}dy,\nonumber
\end{align}
from which we get:

\begin{align}
 \lvert u_x\rvert(x,\tilde{y},t)&\leq \lvert\omega\rvert(x,\tilde{y},t)\left\lvert  \frac{\omega\omega_x-u_x\omega_y}{\omega_y}  \right\rvert_{\infty}\int_{y_0(x)}^{\tilde{y}}\frac{\omega_y}{\omega^2}dy \nonumber \\
&= \left\lvert  \frac{\omega\omega_x-u_x\omega_y}{\omega_y}  \right\rvert_{\infty}\left\lvert\frac{\omega(x,\tilde{y},t)}{\omega(x,y_0(x),t)}-1 \right\rvert\nonumber \\
&\leq \left(\frac{\Vert \omega_0 \Vert_{\infty}}{\min \lvert\omega_0\rvert}+1\right)\left\lvert  \frac{\omega\omega_x-u_x\omega_y}{\omega_y}  \right\rvert_{\infty},  \nonumber 
\end{align}
for every $\tilde{y}\in [0,1]$, which leads to (\ref{uxinfty}).
On the other hand, by Lemmas \ref{materialderivative} and \ref{lemmamaterialomega}:
\begin{align}
D_t\left(\frac{\omega_x^2}{\omega_y}\right)&=\frac{-2u_x\omega_x^2\omega_y-2\omega_xv_x\omega_y^2-u_x\omega_x^2\omega_y+\omega\omega_x^3}{\omega_y^2} \nonumber \\
&=\left(\frac{\omega\omega_x}{\omega_y}-3u_x\right)\frac{\omega_x^2}{\omega_y}-2\omega_xv_x,\nonumber
\end{align}
integrating in $\mathbb{T}\times (0,1)$, we get
\begin{equation}\label{omegaxl2}
\partial_t\int_{\mathbb{T}\times(0,1)}\frac{\omega_{x}^2}{\omega_y}dxdy =\int_{\mathbb{T}\times(0,1)}\frac{\omega_{x}^2}{\omega_y}\left(\frac{\omega\omega_x}{\omega_y}-3u_x\right)dxdy,
\end{equation}
where we have used that
$$\int_{\mathbb{T}\times (0,1)}\omega_xv_xdxdy=\int_{\mathbb{T}\times (0,1)}u_xu_{xx}dxdy=0,$$
thanks to integration by parts and the $x$-periodicity. Thus, by (\ref{omegaxl2}) and (\ref{uxinfty})
\begin{align}
\partial_t\int_{\mathbb{T}\times(0,1)}\frac{\omega_{x}^2}{\omega_y}dxdy &=\int_{\mathbb{T}\times(0,1)}\frac{\omega_{x}^2}{\omega_y}\left(\frac{\omega\omega_x}{\omega_y}-3u_x\right)dxdy \nonumber \\
&\leq  \left(\left\lvert\frac{\omega\omega_x-u_x\omega_y}{\omega_y}  \right\rvert_{\infty}+ 2 \left\lvert u_x  \right\rvert_{\infty}\right)\int_{\mathbb{T}\times(0,1)}\frac{\omega_{x}^2}{\omega_y}dxdy\nonumber \\
& \leq  \left(3+\frac{2\Vert \omega_0 \Vert_{\infty}}{\min \lvert\omega_0\rvert}\right)\left\lvert\frac{\omega\omega_x-u_x\omega_y}{\omega_y}  \right\rvert_{\infty}\int_{\mathbb{T}\times(0,1)}\frac{\omega_{x}^2}{\omega_y}dxdy,\nonumber
\end{align}
By applying Gr\"onwall's inequality, we get (\ref{omegaxl2bound}), which completes the proof of Proposition \ref{omegaxl2prop}.

\end{proof}

Let us state the main result of this section.
\begin{theorem}\label{continuationcriteria}
Assume that $\omega_0\in H^4(\mathbb{T}\times (0,1))$ satisfies the local Rayleigh condition. Let $\omega\in C([0,T];H^4(\mathbb{T}\times(0,1)))$ be the solution to (\ref{VHEE}) that satisfies the local Rayleigh condition in $[0,T]$. Let $E_1$ and $E_2$ be defined by (\ref{defE1}) and (\ref{defE2}) respectively. Assume further that $\lvert\omega_0\rvert >0$ in $\mathbb{T}\times(0,1)$. Then, we have:

\begin{align}
&\left\Vert \frac{1}{\omega_y(t)}\right\Vert_{L^{\infty}(\mathbb{T}\times (0,1))}\leq \left\Vert \frac{1}{\partial_y\omega_0}\right\Vert_{L^{\infty}(\mathbb{T}\times (0,1))}\exp(M(t)),\text{\hspace{0.1cm} for $0\leq t\leq T$,}\label{continuationcriteria1}\\
&\left\Vert \omega_y(t)\right\Vert_{L^{\infty}(\mathbb{T}\times (0,1))}\leq \left\Vert \partial_y\omega_0\right\Vert_{L^{\infty}(\mathbb{T}\times (0,1))}\exp(M(t)),\text{\hspace{0.1cm} for $0\leq t\leq T$,}\label{continuationcriteria2}\\
& \left\Vert \frac{\omega_x(t)}{\sqrt{\omega_y(t)}}\right\Vert_{L^2(\mathbb{T}\times (0,1))}\leq  \left\Vert \frac{\partial_x\omega_0}{\sqrt{\partial_y\omega_0}}\right\Vert_{L^2(\mathbb{T}\times (0,1))}\exp\left(\frac{C(\omega_0)}{2}M(t)\right),\text{ for $0\leq t\leq T$,}\label{continuationcriteria3}\\
&E_1(T)\leq \left\Vert \frac{\partial_x\omega_0}{\sqrt{\partial_y\omega_0}}\right\Vert_{L^2(\mathbb{T}\times (0,1))}\Vert\omega_0\Vert_{\infty}\left\Vert\frac{1}{\partial_y\omega_0}\right\Vert_{L^{\infty}(\mathbb{T}\times (0,1))}\exp\left(\tilde{C}(\omega_0)M(T)\right),\label{continuationcriteria4}\\
&E_2(T)\leq \frac{9}{4}\left\Vert \frac{\partial_x\omega_0}{\sqrt{\partial_y\omega_0}}\right\Vert_{L^2(\mathbb{T}\times (0,1))}\Vert\omega_0\Vert_{\infty}^3\left\Vert\frac{1}{\partial_y\omega_0}\right\Vert_{L^{\infty}(\mathbb{T}\times (0,1))}\exp\left(\tilde{C}(\omega_0)M(T)\right),\label{continuationcriteria5}
\end{align}
where
\begin{equation}
M(t)=\int_{0}^{t}\left\lvert  \frac{\omega\omega_x}{\omega_y}-u_x \right\rvert_{L^{\infty}(\mathbb{T}\times (0,1))}d\tau,
\end{equation}
$C(\omega_0)=3+\frac{2\Vert \omega_0 \Vert_{\infty}}{\min \lvert\omega_0\rvert}$ and $\tilde{C}(\omega_0)=\frac{1}{2}+\frac{C(\omega_0)}{2}$.
\end{theorem}

\begin{proof}

Let $\alpha$ be any point in $\mathbb{T}\times (0,1)$. By (\ref{lemmaformulas1})

\begin{align}
\left\lvert\log(\omega_y(X(\alpha,\tau), Y(\alpha,\tau),\tau))\rvert_{\tau=0}^{\tau=t}\right\rvert&=\left\lvert\int_{0}^{t}\partial_{\tau}\log(\omega_y(X(\alpha,\tau), Y(\alpha,\tau),\tau))d\tau\right\rvert\nonumber\\
&=\left\lvert\int_{0}^{t}\left(\frac{\omega\omega_x}{\omega_y}-u_x\right)(X(\alpha,\tau), Y(\alpha,\tau),\tau)d\tau\right\rvert\nonumber\\
&\leq \int_{0}^{t}\left\lvert  \frac{\omega\omega_x}{\omega_y}-u_x \right\rvert_{L^{\infty}(\mathbb{T}\times (0,1))}d\tau,\nonumber
\end{align}
where $(X(\alpha,\tau), Y(\alpha,\tau))$ is the characteristic curve starting at $\alpha$. Thus, applying exponential we get (\ref{continuationcriteria1}) and (\ref{continuationcriteria2}). From (\ref{omegaxl2bound}) we get (\ref{continuationcriteria3}). By H\"older's inequality, (\ref{continuationcriteria1}) and (\ref{continuationcriteria3})
\begin{align}
\int_{\mathbb{T}\times (0,1))}\frac{\lvert\omega\omega_x\rvert}{\omega_y}dxdy&\leq \left\Vert \frac{\omega_x}{\sqrt{\omega_y}}\right\Vert_{L^2(\mathbb{T}\times (0,1))}\left\Vert \frac{\omega}{\sqrt{\omega_y}}\right\Vert_{L^2(\mathbb{T}\times (0,1))}\nonumber\\
&\leq \left\Vert\omega_0\right\Vert_{L^{\infty}(\mathbb{T}\times (0,1))}  \left\Vert\frac{1}{\partial_y\omega_0}\right\Vert_{L^{\infty}(\mathbb{T}\times (0,1))}^{\frac{1}{2}}\nonumber\\
&\times\exp\left( \frac{M(t)}{2}+\frac{C(\omega_0)M(t)}{2}\right)\left\Vert \frac{\partial_x\omega_0}{\sqrt{\partial_y\omega_0}}\right\Vert_{L^2(\mathbb{T}\times (0,1))},\nonumber\\
\end{align}
which yields (\ref{continuationcriteria4}) and (\ref{continuationcriteria5}) thanks to (\ref{uinfty}). This concludes the proof of Theorem \ref{continuationcriteria}.
\end{proof}
\section{Proof of Theorem \ref{blowup}: Finite time blow-up of the Semi-lagrangian equations}\label{proofblowup}

The aim of this section is to prove Theorem \ref{blowup}, which establishes the finite time blow-up of solutions to the semilagrangian equations (\ref{HSLE}) for certain class of initial data. The following proposition will be useful for proving Theorem \ref{blowup}.

\begin{proposition}\label{formulasSLE}
Let $(\vb{v},h_a)$ be a smooth solution to (\ref{HSLE}). Let $E_1$ and $E_2$ be defined by (\ref{defE1HSLE}) and (\ref{defE2HSLE}) respectively. Then, we have:

\begin{align}
E_1 &=\partial_t \int_{\mathbb{T}^d\times (0,1)}\log(h_a)h_adxda, \label{formulasSLE1} \\
\partial_t E_1&=\int_{\mathbb{T}^d\times (0,1)}\left\lvert\nabla \vb{v}\right\rvert^2h_adxda, \label{formulasSLE2} \\
\partial_t E_2&=\int_{\mathbb{T}^d\times (0,1)}\left\lvert\partial_t \vb{v}\right\rvert^2h_adxda. \label{formulasSLE3}
\end{align}

\end{proposition}

\begin{proof}

By (\ref{HSLE}) and the $x$-periodicity

\begin{align}
\partial_t  \int_{\mathbb{T}^d\times (0,1)}\log(h_a)h_adxda &= \int_{\mathbb{T}^d\times (0,1)}\frac{\partial_t h_a}{h_a}h_a+\log(h_a)\partial_t h_adxda \nonumber \\
&=\int_{\mathbb{T}^d\times (0,1)}-\nabla\cdot(\vb{v}h_a)-\log(h_a)\nabla\cdot (\vb{v}h_a)dxda \nonumber \\
&=\int_{\mathbb{T}^d\times (0,1)}\nabla(\log(h_a))\cdot (\vb{v}h_a)dxda \nonumber \\
&=\int_{\mathbb{T}^d\times (0,1)}\nabla(h_a)\cdot \vb{v}dxda \nonumber \\
&=\int_{\mathbb{T}^d\times (0,1)}(-\nabla\cdot \vb{v})h_adxda, \nonumber 
\end{align}
which proves (\ref{formulasSLE1}). Next, let us compute the time derivative of $E_1$:

\begin{align}
&\partial_t E_1(t) \nonumber \\
&=\int_{\mathbb{T}^d\times (0,1)}\nabla\cdot\nabla\left(\frac{\lvert \vb{v}\rvert^2}{2}+P\right)h_a+\nabla\cdot \vb{v}\nabla\cdot (\vb{v}h_a)dxda \nonumber \\
&=\int_{\mathbb{T}^d\times (0,1)}(\vb{v}\cdot\Delta \vb{v})h_a+(\nabla \vb{v}\cdot \nabla \vb{v})h_a+\nabla\cdot \vb{v}\nabla\cdot (\vb{v}h_a)dxda+\int_{\mathbb{T}^d}\Delta P \int_{0}^{1}h_adadx \nonumber \\
&=\int_{\mathbb{T}^d\times (0,1)} (\vb{v}\cdot\Delta \vb{v})h_a+\left\lvert\nabla \vb{v}\right\rvert^2h_a +\nabla\cdot((\nabla \cdot\vb{v})\vb{v}h_a)-\nabla(\nabla\cdot\vb{v})\cdot\vb{v}h_adxda\nonumber \\
&=\int_{\mathbb{T}^d\times (0,1)} (\vb{v}\cdot\Delta \vb{v})h_a+\left\lvert\nabla \vb{v}\right\rvert^2h_a-(\Delta \vb{v}\cdot \vb{v})h_adxda\nonumber \\
&=\int_{\mathbb{T}^d\times (0,1)}\left\lvert\nabla \vb{v}\right\rvert^2h_adxda, \nonumber
\end{align}
where we have used the identity $\nabla(\nabla\cdot \vb{v})=\Delta \vb{v}$, which follows by the curl free condition $\partial_i\varv_j=\partial_j\varv_i$. Thus, we proved (\ref{formulasSLE2}). Finally, let us compute the time derivative of $E_2$:

\begin{equation}\label{E_2HSLEt}
\partial_t E_2(t)=\int_{\mathbb{T}^d\times (0,1)}\vb{v}_t\cdot \vb{v}_th_a+\vb{v}\cdot \vb{v}_{tt}h_a-\vb{v}\cdot \vb{v}_t\nabla\cdot (\vb{v}h_a)dxda.
\end{equation}
The second term is equal to:

\begin{align}
\int_{\mathbb{T}^d\times (0,1)}\vb{v}\cdot \vb{v}_{tt}h_adxda&=-\int_{\mathbb{T}^d\times (0,1)}\vb{v}\cdot \nabla\left(\vb{v}\cdot \vb{v}_t+P_t\right)h_adxda \nonumber \\
&=\int_{\mathbb{T}^d}P_t\nabla\cdot\int_{0}^{1}\vb{v}h_adadx+\int_{\mathbb{T}^d\times (0,1)}\vb{v}\cdot \vb{v}_t\nabla\cdot (\vb{v}h_a)dxda\nonumber \\
&=\int_{\mathbb{T}^d\times (0,1)}\vb{v}\cdot \vb{v}_t\nabla\cdot (\vb{v}h_a)dxda, \nonumber
\end{align}
thanks to (\ref{divHSLE}), which leads to (\ref{formulasSLE3}). Thus the proof of Proposition \ref{formulasSLE} is now complete.
\end{proof}
Now we can prove Theorem \ref{blowup}.

\begin{proof}[Proof of Theorem \ref{blowup}]

First assume that (\ref{collapseE1SLE}) holds. Integrating in time and applying (\ref{formulasSLE1}), we get:

\begin{align}
d\log\left(\frac{d}{E_1(0)}\frac{1}{\frac{d}{E_1(0)}-t}\right)&=\int_{0}^{t}\frac{d}{\frac{d}{E_1(0)}-\tau}d\tau\nonumber\\
&\leq \int_{0}^{t}E_1(\tau)d\tau\nonumber\\
&=\int_{\mathbb{T}^d\times(0,1)}\log(h_a(\tau))h_a(\tau)dxda\rvert_{\tau=0}^{\tau=t},\nonumber
\end{align}
which implies (\ref{collapselogha}). Next, by (\ref{formulasSLE2}) and H\"older's inequality

\begin{align}
\partial_tE_1&=\int_{\mathbb{T}^d\times (0,1)}\lvert \nabla \vb{v}\rvert^2h_adxda \nonumber\\
&\geq \frac{1}{d}\int_{\mathbb{T}^d\times (0,1)}\lvert \nabla\cdot \vb{v}\rvert^2h_adxda\nonumber\\
&\geq \frac{1}{d}\left(\int_{\mathbb{T}^d\times (0,1)}\lvert \nabla\cdot \vb{v}\rvert h_adxda\right)^2\nonumber\\
&\geq\frac{1}{d} E_1^2.\label{growthE1HSLE}
\end{align}
Thus, Lemma \ref{blowuplemma}, (\ref{formulasSLE2}) and (\ref{growthE1HSLE}) lead to (\ref{collapseE1SLE}). Finally, by (\ref{formulasSLE3}) and H\"older's inequality
\begin{align}
\int_{\mathbb{T}^d\times (0,1)}\lvert \vb{v}\rvert^2h_adxda\cdot\partial_t E_2 &\geq \left(\int_{\mathbb{T}^d\times (0,1)}\lvert \vb{v}\rvert \lvert \vb{v}_t\rvert h_adxda\right)^2 \nonumber \\
&\geq \left(\int_{\mathbb{T}^d\times (0,1)}\vb{v}\cdot \vb{v}_th_adxda\right)^2 \nonumber \\
&=E_2^2.\label{growthE2HSLE}
\end{align}
Thus, Lemma \ref{blowuplemma}, (\ref{growthE2HSLE}), (\ref{formulasSLE3}) and (\ref{kineticHSLE}) lead to (\ref{collapseE2SLE}). This concludes the proof of Theorem \ref{blowup}.
\end{proof}
\section{Appendix. Integral lemma}

The aim of this appendix is to prove the integral identity conerning the integral of the logarithm of a function.

\begin{lemma}\label{loglemma}
Let $(\Omega, \mu)$ be a measure space and let $f\in L^1(\Omega)$ be a positive function in $\Omega$. Assume that $\mu (\Omega)=1$, then

\begin{equation}\label{loglemmaeq}
\exp\left(\int_{\Omega}\log(f)d\mu\right)=\lim_{p\rightarrow 0}\left(\int_{\Omega}f^pd\mu\right)^{\frac{1}{p}}.
\end{equation}

\end{lemma}

\begin{proof}
Let us proceed as in \cite{Gr14}, let $p_n$ be any positive sequence converging to $0$.
By Jensen's inequality,

$$ \exp\left(\int_{\Omega}\log(f)d\mu\right)\leq \left(\int_{\Omega}f^{p}d\mu\right)^{\frac{1}{p}}, \text{\hspace{0.1cm} $\forall p>0$},$$
which implies
\begin{equation}\label{bound1loglemma}
\exp\left(\int_{\Omega}\log(f)d\mu\right)\leq \liminf\left(\int_{\Omega}f^{p_n}d\mu\right)^{\frac{1}{p_n}}.
\end{equation}
Since $\log t\leq t-1$, for $t>0$, we conclude that
\begin{equation}\label{ineq1loglemma}
 \log \left( \int_{\Omega}f^{p_n}d\mu\right) \leq \int_{\Omega}(f^{p_n}-1)d\mu.
 \end{equation}
Multiplying (\ref{ineq1loglemma}) by $\frac{1}{p_n}$ and applying exponential, we get
$$ \left( \int_{\Omega}f^{p_n}d\mu\right)^{\frac{1}{p_n}}\leq \exp\left(  \frac{1}{p_n}\int_{\Omega}\left(f^{p_n}-1\right)d\mu\right),$$
which implies
\begin{equation}\label{bound2loglemma}
\limsup \left( \int_{\Omega}f^{p_n}d\mu\right)^{\frac{1}{p_n}}\leq  \limsup \exp\left(  \frac{1}{p_n}\int_{\Omega}\left(f^{p_n}-1\right)d\mu\right).
\end{equation}
Now, let us compute $ \limsup \exp\left(  \frac{1}{p_n}\int_{\Omega}\left(f^{p_n}-1\right)d\mu\right)$. Assume that $q_n\searrow 0$ , then
\begin{equation}\label{convergenceloglemma}
\frac{1}{q_n}\left(t^{q_n}-1\right)\searrow \log(t), \text{\hspace{0.1cm} $\forall t>0$ }.
\end{equation}
Let $h_n:\Omega\rightarrow \mathbb{R}$ be a sequence of positive functions defined by

$$ h_n(x)=\frac{1}{q_0}\left(f^{q_0}-1\right)-\frac{1}{q_n}\left(f^{q_n}-1\right).$$
By (\ref{convergenceloglemma}),
$$h_n(x)\nearrow \frac{1}{q_0}\left(f^{q_0}-1\right)-\log(f).$$
Applying Lebesgue's monotone convergence theorem, we get

$$ \int_{\Omega}h_nd\mu\nearrow \int_{\Omega}\frac{1}{q_0}\left(f^{q_0}-1\right)-\log(f) d\mu.$$
Thus, 
$$\int_{\Omega}\frac{1}{q_n}\left(f^{q_n}-1\right)d\mu \searrow \int_{\Omega}\log(f)d\mu.$$
Therefore,

\begin{equation}\label{limitloglemmaeq}
 \limsup \exp\left(  \frac{1}{p_n}\int_{\Omega}\left(f^{p_n}-1\right)d\mu\right)=\exp\left( \int_{\Omega}\log(f)d\mu\right).
\end{equation}
Finally, by (\ref{bound1loglemma}), (\ref{bound2loglemma}) and (\ref{limitloglemmaeq})

$$ \limsup \left(\int_{\Omega}f^{p_n}d\mu\right)^{\frac{1}{p_n}}\leq\exp\left(\int_{\Omega}\log(f)d\mu\right)\leq \liminf\left(\int_{\Omega}f^{p_n}d\mu\right)^{\frac{1}{p_n}}.$$
This concludes the proof of Lemma \ref{loglemma}.

\end{proof}

\section*{Acknowledgements}

The author thanks Yann Brenier for many interesting discussions (in particular for mentioning the possible extension to higher dimensions of the blow-up result) and for his hospitality during the author's stay at Laboratoire de Math\'ematiques d'Orsay.

This work is part of the grant SEV-2015-0554-17-4 funded by: MCIN/AEI/ 10.13039/501100011033.

\bibliography{biblio} \bibliographystyle{alpha}

\end{document}